\documentclass[11pt]{article}

\usepackage{natbib}

\usepackage[a4paper, margin=1.2in]{geometry}

\DeclareMathAlphabet{\mathpzc}{OT1}{pzc}{m}{it}

\usepackage[dvipsnames]{xcolor}

\usepackage[english]{babel}
\usepackage[T1]{fontenc}
\usepackage{ae}
\usepackage{fancyhdr}
\usepackage[utf8]{inputenc}
\usepackage{amsmath}

\usepackage{amsthm}
\usepackage{amssymb}   
\usepackage{amsfonts}
\usepackage{mathtools} 
\usepackage{marvosym}
\usepackage{extarrows}
\usepackage{bm}
\usepackage{subfigure}
\usepackage{xcolor}
\usepackage{dsfont}
\usepackage{enumitem}
\usepackage{diagbox}
\usepackage{makecell}
\usepackage[table]{xcolor}
\usepackage[toc,page]{appendix}
\usepackage{array}
\usepackage{caption}
\usepackage{blindtext}
\usepackage[normalem]{ulem}

\usepackage{graphicx}    
\usepackage{psfrag}
\usepackage{rotating}

\usepackage{comment} 

\usepackage{url}

\usepackage[hidelinks]{hyperref}
\hypersetup{
     colorlinks   = true,
     linkcolor    = blue,
     citecolor    = blue,
     urlcolor	  = blue
}

\theoremstyle{plain} 
\newtheorem{Theorem}{Theorem}[section]

\theoremstyle{plain} 
\newtheorem{Lemma}[Theorem]{Lemma}
	
\theoremstyle{plain}   
\newtheorem{Proposition}[Theorem]{Proposition}

\theoremstyle{plain}

\theoremstyle{plain}

\theoremstyle{plain}
\newtheorem{Remark}[Theorem]{Remark} 

\theoremstyle{definition}

\theoremstyle{plain}

\numberwithin{equation}{section}

\newcommand{\R}{\mathbb{R}} 
\newcommand{\N}{\mathbb{N}} 
\newcommand{\Pro}{\mathbb{P}} 
\newcommand{\E}{\mathbb{E}} 

\newcommand{\indep}{\perp\!\!\!\perp}

\def\pd#1su#2.{\frac{\partial #1}{\partial #2}}
\def\td#1su#2.{\frac{d #1 }{d #2}}
\def\md#1.{\frac{D #1 }{ Dt }}

\makeindex

\begin{document}



\title{\bf Before and beyond the mixing time: \\ New approximations for additive functionals of stationary Gauss-Markov processes\footnote{Project funded by the Deutsche Forschungsgemeinschaft (DFG) – Project-ID 499552394 – CRC 1597\\
\indent $~^{1}$gabriele.bellerino@stochastik.uni-freiburg.de, $^2$angelika.rohde@stochastik.uni-freiburg.de
}}
\date{}
\author{Gabriele Bellerino$^1$ and Angelika Rohde$^{2}$\\
Mathematical Institute, University of Freiburg}
\maketitle

\vspace{-6mm}
\begin{abstract}
Whereas classical invariance principles for ergodic Markov chains address the situation in which the time horizon of observations is large, the quality of approximation is questionable when this is not the case anymore --  even when starting the Markov chain in the invariant law.  In this article, we prove quantitative and functional limit theorems for additive functionals along triangular arrays of stationary Gaussian Markov processes when the mixing time $t_{\text{mix}}$ scales sub-, super- and proportionately to the number of observations $n$. Our major finding is a phase transition
at $t_{\text{mix}}\asymp n$, together with the identification and interrelation properties of the emerging new limit processes at and before the mixing time.
\end{abstract}

\section{Introduction and main result}
\label{Introduction}

\smallskip

Additive functionals $S_n(f)=\sum_{k=1}^nf(X_k)$ of stochastic processes $(X_k)$ constitute one of the most fundamental objects in probability theory. The literature on this topic is extensive and spans a wide range of probabilistic frameworks, including general dependency structures (\cite{DurrettResnick1978}, \cite{Davis1983}, \cite{DehlingEtAl1986}), self-similar Gaussian processes (\cite{HongEtAl2024}) and predominantly Markov processes; see \cite{BhattacharyyaLee1988}, 
\cite{Jain1990}, \cite{PachecoPrabhu1996}, \cite{NeyAcosta1998}, \cite{Chen1999},  \cite{MaxwellWoodroofe2000}, \cite{JaraEtAl2009} for the discrete case and \cite{GetoorSharpe1973}, \cite{Maisonneuve1975}, \cite{Nevison1976}, \cite{Dynkin1977}, \cite{Glover1980},  \cite{Glover1981}, \cite{Glover1981_0}, \cite{Kaspi1988}, \cite{Dynkin1991},  \cite{MarcusRosen1996}, \cite{ChenetAl2008},  \cite{Kuwae2010}, \cite{ChenEtAl2019},  \cite{DeuschelEtAl2021} for time-continuous processes and the corresponding integral functional.

Classical (functional) limit theorems for positive recurrent Markov chains describe the behavior of the (sequential) additive functional when the time horizon of observations is large. Indeed, if the Markov chain is stationary ergodic and $f$ is square-integrable with respect to the invariant measure, the sequential limit is a Brownian motion, under additional assumptions such as the existence of a solution to the Poisson equation  (\cite{BhattacharyyaLee1988}) or under conditions on the moments of $f$ and the growth of the conditional mean $\E [S_n(f) | X_0 ]$ (\cite{MaxwellWoodroofe2000}), the latter within a logarithmic term of being necessary. However, if the time horizon is not large as often present in physical applications, the distributional approximation by the Brownian motion seems to be inadequate, even in case of ergodic Markov chains starting in their invariant law. 

To approach the small-sample effect in an asymptotic framework, we study the additive functional along a triangular array of stationary processes with diverging mixing times. 
For technical convenience, we carry out our analysis  in the context of discrete-time stationary Gaussian Markov processes, which  reduces to the fundamental example of  autoregressive processes (cf.~Lemma \ref{lemma:AR1}).   On the one hand, stationary autoregressive processes which are not constant in time are geometrically ergodic and represent the classical setup for near-unit-root analysis (cf.~\cite{ChanWei1987},  \cite{Cox1991}, \cite{CoxLlatas1991},  \cite{BhattacharyyaRichardsonFranklin1997}, and \cite{BuchmannChan2007}), while on the other hand, the theory of Gaussian processes provides well-developed mathematical tools which allow for a profound analysis, leaving technical issues aside. 
To this end, let $X_{1,n},\dots, X_{n,n}$ be consecutive observations of a stationary time series satisfying the recursion  
\begin{align}
	X_{j,n} = \alpha_{n} X_{j-1,n} + \sigma_n\varepsilon_{j}  \quad \text{for }j\in\mathbb{N},\label{Z_intro}
\end{align}
where $(\varepsilon_j)$ are iid $\mathcal{N}(0,1)$. Subsequently, $\pi_n = \mathcal{N} \big( 0 , \sigma_n^2 / (1 - \alpha_n^2) \big)$ denotes its invariant distribution and  $t_{\text{mix},n}$ the local mixing time in total variation distance $d_{TV}$, defined as
\begin{align}
	t_{\text{mix},n} ( \varepsilon, \delta ) := \inf \bigg\{ m \geq 0 :  \sup_{ |x| \leq k_n(\delta)}    d_{TV} \Big( \mathcal{L}_x(X_{m,n}), \pi_n \Big) \leq \varepsilon  \bigg\} \ \text{ for }\varepsilon,\delta\in (0,1), \label{eq:t_mix}
\end{align}
with $k_n(\delta) = \inf \{k > 0 : \pi_n \left( \left[ -k,k \right] \right) \geq 1 - \delta \} $. Here, the subscript $x$ in the law $\mathcal{L}_x(X_{m,n})$ of $X_{m,n}$ indicates that the process is started at $X_{0,n}=x$. If nothing is explicitly stated, we tacitly assume $X_{0,n}\sim\pi_n$. The qualitative behavior of the dynamics  \eqref{Z_intro} depends crucially on how close the regression parameter $\alpha_n$ is to a unit root, but likewise how the innovations' variance $\sigma_n^2$ is related to  $\alpha_n$. For instance,
\begin{itemize}
	\item if $\sigma_n^2=\sigma^2$ is constant while $\alpha_n \nearrow 1$, the process approaches a random walk and the recurrence behavior transitions from positive recurrent to null-recurrent.
	
	\smallskip
	\item If $\sigma_n^2=1-\alpha_n^2$ while $\alpha_n \nearrow 1$, the recursion \eqref{Z_intro} preserves the fixed invariant distribution $\mathcal{N}(0,1)$ and approaches the constant relation $X_t=X_{t-1}$, revealing a particular instance of long-range dependent dynamics, and the memory behavior of the process transitions from short-range to long-range dependence. 
\end{itemize} 
As in the classical literature for near-unit-root analysis of AR(1) processes, we parametrize $ \alpha_n=\left(1-\frac{\gamma}{n^{\beta}}\right)$ for some constants  $\gamma\in (0,1)$ and $\beta>0$. Note that although null-recurrence  and long-range dependence are two significantly different probabilistic qualities,  the mixing time scales as $$t_{\text{mix},n}\asymp \frac{n^\beta}{\gamma} $$ in both scenarios (cf. Lemma \ref{lemma:t_mix}). For $q \in \N_0= \N \hspace{0.3mm} \cup \{0\}$, let $H_q$ denote the Hermite polynomial of order $q$. It is well-known that $L^2(\mathcal{N}(0,1))$ is a separable Hilbert space with orthonormal basis $\big\{ H_q\big/\sqrt{q!}: q \in \N_0 \big\}$. 
Let ${S}_{n,t}(f)=\sum_{k=1}^{\lfloor nt\rfloor}f(X_{k,n})$ for $t\in [0,1]$, and define 
\[
\hat{S}_{n,t}(f)= \frac{S_{n,t}(f) - \E S_{n,t}(f)}{\sqrt{\mathrm{Var} S_{n,1}(f)}}, \qquad t \in [0,1],
\]
as the corresponding standardized sequential process. Our main result states quantitative total variation approximation of marginals and convergence in distribution in the Skorokhod space $(D[0,1], d_S)$ for $$ \big( \hat{S}_{n,t}(f) \big)_{t \in [0,1]}$$ along the triangular array~\eqref{Z_intro} with $ \alpha_n=\left(1-\frac{\gamma}{n^{\beta}}\right)$ for some constants $\gamma\in (0,1)$, $\beta>0$ (Sections \ref{sec: beta<1}, \ref{sec: beta=1}, and \ref{sec: beta>1}). As total variation constitutes a central quantity in statistical decision theory, explicit bounds in total variation are very relevant from an information-theoretic perspective. The limit processes for a fixed polynomial $f=\sum_{q=m}^p c_q H_q$ are displayed in Figure~\ref{fig}. Here, $Z\sim\mathcal{N}(0,1)$, $B$ denotes  a Brownian motion, and the processes $W_{q,\gamma}$ are defined by the spectral domain representation
\begin{align*}
	W_{q,\gamma}(t)	= K_q^{(\gamma)} \int_{\R^q}^{''} \frac{e^{i \left( \sum_{r=1}^q \lambda_r \right) t} -1}{i \left( \sum_{r=1}^q \lambda_r \right)} \frac{1}{\left(2 \pi \right)^{q/2}} \prod_{r=1}^{q} \frac{1}{i \lambda_r + \gamma}  \mathrm{d} \tilde{B}_{0} \left( \lambda_1 \right) \ldots \mathrm{d} \tilde{B}_{0} \left( \lambda_q \right) , \quad t \in [0,1],    \end{align*}
where $K_q^{(\gamma)} \in \R_{\geq 0}$, $ \int_{\R^q}^{''}$ denotes integration outside the diagonals $\lambda_i=\pm \lambda_j$, $i\not=j$ and  $ \tilde{B}_{0}$ is the complex Gaussian white noise random measure (cf.~Section 9 in \cite{PeccatiTaqqu2011}).  For two positive sequences $(x_n)_{n \in \N}, (y_n)_{n \in \N}$, we write $x_n \gg y_n$ if $y_n=o(y_n)$.

\begin{figure}[h]
	\includegraphics[width=1.0\linewidth]{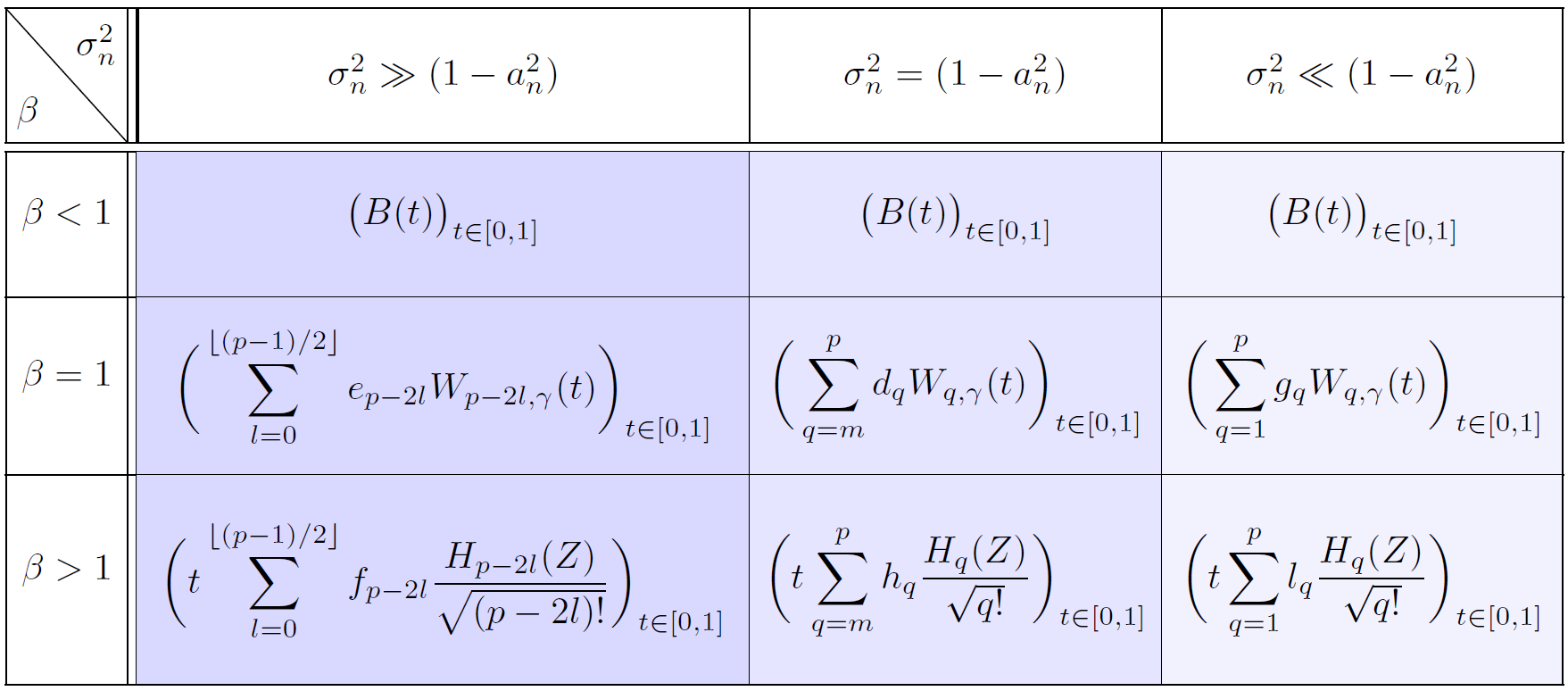}
	\caption{
		Limit processes of $ \big( \hat{S}_{n,t}(f) \big)_{t \in [0,1]}$ for $f=\sum_{q=m}^p c_q H_{q}$ along  the triangular array~\eqref{Z_intro}. \\$ d_q=d_q(f), e_q=e_q(f)$, $ f_q=f_q(f)$, $ g_q=g_q(f)$, $ h_q=h_q(f)$, $ l_q=l_q(f)$ are real constants. 
	} \label{fig}
\end{figure}
Note that for all ranges of $\sigma_n^2$, the phase transition phenomenon at $\beta=1$ exactly mirrors the relation of mixing time and time horizon of observations: 
$n \gg t_{\text{mix},n}$ if $\beta<1$, $n \asymp t_{\text{mix},n}$ if $\beta=1$ and $n \ll t_{\text{mix},n}$ if $\beta>1$. Correspondingly, the limit is a Brownian motion as in classical asymptotics for $\beta<1$, but it is significantly different for $\beta \geq 1$.

\medskip
\noindent
{\bf Beyond the phase transition.} From the phase transition at $\beta=1$  on, the parameter $\sigma_n^2$ starts to matter. This is getting especially apparent if the additive functional is not evaluated at a fixed polynomial $\sum_{q=m}^p c_q H_q$ as in Figure \ref{fig} but at $f_n=\sum_{q=m}^p c_q \delta_n^q H_q(\cdot/\delta_n)$ for some family $(c_q)$ and $\delta_n^2=\sigma_n^2/(1-\alpha_n^2)$, i.e.~$\mathcal{N}(0,\delta_n^2)=\pi_n$. 
As Theorems~\ref{theorem:beta<1_fn}, \ref{theorem:beta=1_fn}, and \ref{theorem:beta>1_fn} reveal, for $\sigma_n^2 \gg  (1 - \alpha_n^2) $,  the asymptotics is then governed by the highest Hermite coefficient (order-driven regime). For $\sigma_n^2 = \sigma^2 (1 - \alpha_n^2) $, all Hermite coefficients contribute  to the limit. If $\sigma_n^2 \ll  (1 - \alpha_n^2) $, a classical pattern re-emerges for this choice of $f_n$, with the asymptotics solely determined by the Hermite rank. 
Observe that in both two last regimes, the recursion \eqref{Z_intro} approaches the constant relation  $X_t=X_{t-1}$, indicating (degenerate) long-range dependence, but the Hermite-rank as a key quantity arises exclusively in the latter one. However, it evolves here to a degenerate process from the phase transition on, bypassing the traditional Dobrushin-Major-Taqqu long-memory limit regime.

\smallskip
\paragraph*{Preview of the proof} In view of the targeted variation bounds, the proof is conducted as follows.

\begin{itemize}
	\item $\beta<1$: 
	The proof relies on the Wiener chaos decomposition of the additive functional, expressed in terms of multiple Wiener-It\^{o} integrals together with a quantitative version of the Fourth-Moment theorem (\cite{NualartPeccati2005}, \cite{NourdinPeccatiReinert2010}). The contractions in the remainder are explicitly evaluated and bounded by means of symmetry considerations, a log-convexity argument  and Chebychev's sum inequality. 
	\smallskip
	\item $\beta=1$: 
	For the triangular array of AR(1)-observations $( X_{1,n}, \ldots,X_{n,n} ), \, n \in \N$, we exploit the spectral representation, building upon the seminal ideas of \cite{DobrushinMajor1979} originally developed in the context of long-range dependent time series. Using the self-similarity of complex Brownian motion within the multiple Wiener-It\^{o} representation of $\hat{S}_{n,t}(H_q)$ and the isometry, the result is traced back  to proving convergence of the respective kernels in $L^2$. Here, a sharp bound is needed, which we build especially  on Jordan's inequality. The multivariate extension over different chaos orders employs the result of \cite{BaiTaqqu2018}.
	\smallskip
	\item $\beta>1$: 
	The proof relies on convergence of the finite-dimensional marginals in $L^2$,  taking advantage of the recursive structure and orthogonality relations inherent to the Wiener chaos decomposition.
\end{itemize}

\smallskip
\paragraph*{\bf{Limit process for}  {\boldmath$\beta = 1$}}

The processes $W_{q,\gamma}$  in the limit for $\beta=1$ are so-called tempered Hermite processes as introduced in \cite{Sabzikar2015}, motivated by modeling of turbulence. As such, they are continuous with stationary increments, not self-similar and possess the equivalent time domain representation as multiple Wiener-It\^o integral 
\begin{align}
	W_{q,\gamma}(t) =  K_q^{(\gamma)}  \int_{\R^q}^{'} \left(\int_{0}^{t} \prod_{r=1}^{q} e^{-\gamma (s-x_r)} \, \mathbf{1}_{\{ x_r \leq s \}} \, \mathrm{d}s \right) \mathrm{d} B(x_1) \cdots \mathrm{d} B(x_q), \quad t \in [0,1], \label{eq:1.2}
\end{align}
with prime indicating the integration off the diagonals (i.e.~$x_i=x_j$, $i\not= j$).
Interestingly, the limit processes for $\beta=1$ depend continuously on $\gamma$ with respect to the topology of weak convergence, and reveal the limiting processes for $\beta<1$ (Brownian motion) and $\beta>1$ as $\gamma \to \infty$ and $\gamma \to 0$, respectively (Theorem \ref{theorem:continuity}). Note that the previous restriction $\gamma \in (0,1)$ was only made to guarantee that \eqref{Z_intro} is stationary for \emph{every} $n \in \N$. For the component $W_{q,\gamma}$, we also provide explicit bounds on the total variation distance of the marginals.

\begin{Theorem}[Total variation bound]
	\label{prop:gamma}
	For any $q \in \N$, there exists a positive constant $C(q)$ such that the following bounds are satisfied.
	\begin{align*}
		(i) &\quad d_{TV} \left( \mathcal{L}\big(W_{q,\gamma}(1)\big) , \mathcal{N}(0,1) \right) \leq C(q) \gamma^{-1/2}  (1+a_q(\gamma)), \qquad \hspace{10mm}\gamma \geq 1,\\
		(ii) &\quad d_{TV} \left( \mathcal{L}\big(W_{q,\gamma}(1)\big) , \mathcal{L} \left(  \frac{H_q(Z)}{\sqrt{q!}} \right) \right) \leq  C(q) \gamma^{1/(4q)} (1+b_q(\gamma)) ,
	\end{align*}
	where $a_q(\gamma) \to 0$ as $\gamma \to \infty$ and $b_q(\gamma) \to 0$ as $\gamma \to 0$ for every $q \in \N$.
\end{Theorem}

\paragraph*{\bf{The boundary case} {\boldmath{$\sigma_n^2 = 1 - \alpha_n^2$}}}
If $\sigma_n^2= 1 - \alpha_n^2$ in the recursion \eqref{Z_intro}, the invariant distribution $\pi_n$ does not depend on $n$ and equals $\mathcal{N} (0,1)$. In this case, the results displayed in Figure \ref{fig} even apply to any $f \in L^2 ( \mathcal{N} (0,1) )$ (corresponding to $p=\infty$ with the series defined as limits in $L^2$, respectively). The extension to $f \in L^2 (\mathcal{N} (0,1) )$ builds on the sharp maximal inequality for stationary sequences of \cite{PeligradUtevWu2007}.  Moreover, the convergence of the finite-dimensional marginals for $\beta>1$ even holds in $L^2$  when the recursions start in a fixed random variable $Z \sim \mathcal{N}(0,1)$ (cf.~Section~\ref{sec: beta>1}). 

Finally, we also provide explicit bounds on the total variation distance between the distribution of  $\hat S_{n,1}(H_q)$ and the marginals of the corresponding limit processes. We remark that in consonance with the heuristics, the bound in (i) is getting looser as $\beta \nearrow 1$ and likewise, the bound in (iii) is getting tighter as $\beta \nearrow \infty$, as far as the dependence on $n$ is concerned. Below, $x(n,\gamma)=o_n(1)$ means $\lim_{n\rightarrow \infty}x(n,\gamma)=0$, for every $\gamma\in(0,1)$.

\begin{Theorem}[Quantitative limit theorems]\label{prop: QLT}
	Let $\sigma_n^2 = 1 - a_n^2$ with $ \alpha_n=\left(1-\frac{\gamma}{n^{\beta}}\right)$ for some constants $\gamma\in (0,1)$. For any $q\in\mathbb{N}$, there exists a constant $C(q)>0$ and a function $\gamma \mapsto f_q(\gamma) > 0$ such that for all $n\in\mathbb{N}$ and $\beta>0$ the following bounds hold true. 
	\begin{align*}
		(i) &\quad d_{TV} \left( \mathcal{L}\big(\hat{S}_{n,1}(H_q)\big) , \mathcal{L}(B(1)) \right) \leq C(q) \gamma^{- \frac{1}{2}} n^{\frac{\beta-1}{2}} (1+o_n(1)),  &\text{for any } \beta \in (0,1);\\
		(ii) &\quad d_{TV} \left( \mathcal{L}\big(\hat{S}_{n,1}(H_q)\big) , \mathcal{L}(W_{q,\gamma}(1)\big)  \right) \leq   f_q(\gamma) \, n^{-\frac{1}{9q}}  (1+o_n(1)),   &\text{for } \beta=1;\\
		(iii) &\quad d_{TV} \left( \mathcal{L}\big(\hat{S}_{n,1}(H_q)\big) ,   \mathcal{L}\Big(\frac{H_q \left( Z   \right) }{\sqrt{q!}}\Big) \right) \leq C(q) \gamma^{\frac{1}{4q}} n^{ \frac{1-\beta}{4q}} (1+o_n(1)),  &\text{for any } \beta > 1.
	\end{align*}
\end{Theorem}
In (i) and (iii), the dependence of the leading coefficient in $\gamma$ exhibits a symmetry consistent with the continuity properties of the limit process $W_{q,\gamma}$ (see Theorem \ref{prop:gamma}).

\paragraph*{\bf{Outline}} The article is organized as follows. Section \ref{sec: preparatory} introduces preliminary variance approximations, essentials of Gaussian Hilbert spaces, and includes the proof of tightness in $D[0,1]$. Sections \ref{sec: beta<1}--\ref{sec: beta>1} contain the formulation and derivation of our main results, in particular the limits of Figure~\ref{fig}. Here, Section~\ref{sec: beta<1} is devoted to $\beta<1$ (Theorems \ref{theorem:beta<1}, \ref{theorem:beta<1_fn} and the proof of Theorem \ref{prop: QLT} (i)), Section~\ref{sec: beta=1} contains the results for  $\beta=1$ (Theorems \ref{theorem:beta=1}, \ref{theorem:beta=1_fn} and the proof of Theorem \ref{prop: QLT} (ii)), and Section~\ref{sec: beta>1} treats the case  $\beta>1$ (Theorems~\ref{theorem:beta>1}, \ref{theorem:beta>1_fn} and the proof of Theorem \ref{prop: QLT} (iii)). The proof of the limits' continuity (Theorem \ref{theorem:continuity} and proof of Theorem~\ref{prop:gamma}) is given in Section~\ref{sec: continuity}.

\section{Preparatory results}
\label{sec: preparatory}

\subsection{Asymptotic variances, covariances and leading terms of $\hat{S}_{n,t}(f)$}

In this subsection, we derive explicit expressions for the asymptotic variance of the additive functional for a polynomial $f$ in case $\sigma_n^2\gg(1-\alpha_n^2)$ and $\sigma_n^2\ll (1-\alpha_n^2)$, and in general for any $f\in L^2(\mathcal{N}(0,\sigma^2))$ if $\sigma_n^2=\sigma^2(1-\alpha_n^2)$ as in this case, the invariant distribution $\mathcal{N}(0,\sigma^2)$ does not depend on $n$. The transformed Hermite polynomials $$  H_q \left( x ;  \sigma^2 \right)   =\sigma^{q} H_q (x/\sigma), \qquad q\in\N_0, $$  form an orthogonal basis of $L^2(\mathcal{N}(0,\sigma^2))$.  To simplify notation, let 
\begin{align}
	 Z_{j,n} = \frac{X_{j,n}}{ \sqrt{\mathrm{Var} X_{j,n} }} \quad (\sim \mathcal{N} (0,1)) \label{eq:Z_j}
\end{align}
and
$$\Lambda_q(n,t) = \begin{cases}
	\displaystyle\sum_{i=1}^{ {\lfloor nt \rfloor} } \sum_{j=1}^{ {\lfloor nt \rfloor} } \left( \E \left[  Z_{i,n}  Z_{j,n} \right] \right)^q & \ \ \text{ for } q \geq 1; \\
	0 & \ \ \text{ for } q = 0.
\end{cases}
$$
Using that $ Z_{i,n}=\alpha_n^i Z_{0,n} +  \sqrt{1 - \alpha_n^2} \sum_{l=1}^i  \alpha_n^{i-l} \, \varepsilon_l $,  elementary calculation reveals for $q \geq 1$
\begin{align}
	\Lambda_q(n,t) = \begin{cases}
		2 \frac{n^{\beta+1}  }{q \gamma} t \left( 1 + o_{n,\text{unif}}(1) \right)    & \text{if } \beta < 1 , \\  
		2 \frac{n^{2}}{q^2 \gamma^2} \left( \exp(-q \gamma t) - 1 + q\gamma t  \right) (1 + o_{n,\text{unif}} (1) )  & \text{if } \beta = 1 , \\
		n^2 t^2 \left( 1 + o_{n,\text{unif}} (1) \right)  & \text{if } \beta > 1 , \label{eq:2_Lambda}
	\end{cases}
\end{align}  
i.e.~for fixed $\beta$ the order in $n$ of the double sum does not depend on $q$. Here, $x_n(t) = o_{n,\text{unif}}(1)$ means $\sup_{t \in [0,1]} |x_n(t)| = o_n(1)$.
Recall that ${S}_{n,t}(f)=\sum_{k=1}^{\lfloor nt\rfloor}f(X_{k,n})$.

\begin{Lemma}\label{lemma: variance}
	\begin{itemize}
		\item[(i)] If $\sigma_n^2 \gg  (1 - \alpha_n^2) $ consider $f=\sum_{q=m}^p c_q H_q$. Then 
		\begin{align*}
			\hat{S}_{n,t}(f) &= c_p \hat{S}_{n,t}(H_p)  + R_{n,t}(f), 
		\end{align*}
		where $R_{n,t}(f) \longrightarrow_{L^2} 0$ for $n \to \infty$. Moreover,  with $C_{p,l}=\frac{p!}{l! (p-2l)! 2^l}$,  \begin{align*}
			\hat{S}_{n,t}(H_p) &= \frac{  \big( \mathrm{Var} X_{1,n} \big)^{p/2} \sum_{i=1}^{\lfloor nt \rfloor}   \sum_{l=0}^{\lfloor (p-1)/2 \rfloor} C_{p,l} H_{p-2l}  \big( Z_{i,n} \big)}{\sqrt{\mathrm{Var}S_{n,1}(H_p)}} \big(1 + o_{n,\mathrm{unif}}(1) \big) \quad\text{and}\\
			\mathrm{Var} S_{n,t}(H_p) &=  \left( \mathrm{Var} X_{1,n} \right)^{p}  \sum_{l=0}^{\lfloor p/2 \rfloor}  C_{p,l}^2 (p-2l)! \Lambda_{p-2l}(n,t)\big(1 + o_{n,\mathrm{unif}}(1) \big) . 
		\end{align*}
		\item[(ii)] If $\sigma_n^2 = \sigma^2 (1 - \alpha_n^2) $ consider $f = \sum_{q=m}^\infty c_q H_q ( \cdot; \sigma^2) \in L^2 \left(\mathcal{N} \left(0, \sigma^2 \right) \right)$. Then 
		\begin{align*}
			\mathrm{Var} S_{n,t}(f) = \sum_{q=m}^\infty  c_q^2 q!  \left(   \sigma^2  \right)^q \Lambda_{q}(n,t) 
		\end{align*}  
		for any $t\in [0,1]$.
		\item[(iii)] If $\sigma_n^2 \ll  (1 - \alpha_n^2) $ consider  $f=\sum_{q=m}^p c_q H_q$. Then 
		\begin{align*}
			\hat{S}_{n,t}(f)  = \frac{ \left( \mathrm{Var} X^{(n)} \right)^{m_*/2}  \sum_{i=1}^{ {\lfloor nt \rfloor} } \sum_{r=1}^{m_*} D_r^{(m_*,m,p)}     H_{r}(Z_{i,n})  }{\sqrt{\mathrm{Var} S_{n,1}(f)}} + R_{n,t}(f),
		\end{align*}
		where $R_{n,t}(f) \longrightarrow_{L^2} 0$ for $n \to \infty$, 
		\begin{align*}
			D_r^{(m_*,m,p)}  =   \sum_{s=0}^{\lfloor p/2 \rfloor}  \mathds{1}_{A_{s,r}} (m_*) \sum_{u=0}^{\lfloor p/2 \rfloor}  \mathds{1}_{B_{m,p}}(m_*+2u) \,  c_{m_*+2u} \,  C_{m_*+2u,s+u} \,  \binom{s+u}{s} (-1)^{u} ,
		\end{align*}
		where the sets are $A_{s,r}= \{ x : x=2s+r \}$ and $B_{m,p} = \{ y : m \leq y \leq p \}$, and  $m_* := \min \{w \geq 1:  \sum_{r=1}^p D_r^{(w,m,p)}  \neq 0 \}$.  Moreover, for any $t\in [0,1]$,
		\begin{align*}
			\mathrm{Var} S_{n,t}(f) = \left( \mathrm{Var} X_{1,n} \right)^{m_*} \sum_{r=1}^{m_*} \left( D_r^{(m_*,m,p)}  \right)^2 r! \; \Lambda_{r}(n,t) \big(1 + o_{n,\mathrm{unif}}(1) \big)  .
		\end{align*}
	\end{itemize}
\end{Lemma}

\begin{proof}
	In case (i), the multiplication theorem for Hermite polynomials yields
	\begin{align}
		H_q \left(X_{i,n}\right) = \sum_{l=0}^{\lfloor q/2 \rfloor} C_{q,l} \left(  \sqrt{\mathrm{Var} X_{1,n} } \right)^{q-2l} \left( \mathrm{Var} X_{1,n} - 1 \right)^l H_{q-2l}  \left( Z_{i,n} \right), \label{mult}
	\end{align}
	where $C_{q,l}=\frac{q!}{l! (q-2l)! 2^l}$.
	Since $\mathrm{Var} X_{1,n} \to \infty$ as $n\rightarrow\infty$, we get for any $l \in \{0,\ldots, {\lfloor q/2 \rfloor} \}$,  
	\begin{align*}
		\left(  \sqrt{\mathrm{Var} X_{1,n} } \right)^{k-2l} \left( \mathrm{Var} X_{1,n} - 1 \right)^l \asymp  \left( \mathrm{Var} X_{1,n} \right)^{k/2} (1+o_n(1)).
	\end{align*}
	Thus, $   H_q \big(X_{i,n}\big) =  \big( \mathrm{Var} X_{1,n} \big)^{q/2} \sum_{l=0}^{\lfloor q/2 \rfloor} C_{q,l} H_{q-2l}  \big( Z_{i,n} \big) (1+o_n(1)) $. For $f=\sum_{q=m}^p c_q H_q$, the Hermite polynomial of order $p$ provides the leading term of the variance. Then the orthogonality of Hermite polynomials, \eqref{eq:2_Lambda} and $\mathrm{Var} \big(H_q( Z_{1,n} ) \big) = q! $ lead to the first identity together with
	\begin{align*}
		\mathrm{Var} S_{n,t}(f) &= c_p^2  \left( \mathrm{Var} X_{1,n} \right)^{p}   \sum_{l=0}^{\lfloor p/2 \rfloor}  C_{p,l}^2 \, (p-2l)! \,  \Lambda_{p-2l}(n,t)\big(1 + o_{n,\mathrm{unif}}(1) \big)  \\
		&= c_p^2 \, \mathrm{Var} S_{n,t}(H_p) \big(1 + o_{n,\mathrm{unif}}(1) \big) .
	\end{align*}
	In case (ii), $ \mathrm{Var} X_{1,n} = \sigma^2$ and the claim follows from  \eqref{eq:2_Lambda} and
	\begin{align*}
		\mathrm{Var} S_{n,t}(f) = \sum_{q=m}^\infty  c_q^2 q!  \left(   \mathrm{Var} X_{1,n}    \right)^q  \sum_{j=1}^{ {\lfloor nt \rfloor} } \sum_{l=1}^{ {\lfloor nt \rfloor} }  \left( \E \left[     Z_{j,n}    Z_{l,n}     \right] \right)^q.
	\end{align*}
	In case (iii), using \eqref{mult} again, but observing this time $\mathrm{Var} X_{1,n} \to 0$ for $n \to \infty$, we have $\big( \mathrm{Var} X_{1,n} - 1 \big)^l = \sum_{s=0}^l \binom{l}{s} (\mathrm{Var} X_{1,n} )^s (-1)^{l-s}$, so it holds 
	\begin{align*}
		f\left(X_{i,n}\right)  =   \sum_{q=m}^p     \sum_{l=0}^{\lfloor q/2 \rfloor}  \sum_{s=0}^l c_q \,  C_{q,l}  \binom{l}{s} (-1)^{l-s}  \left(  \mathrm{Var} X_1^{(n) } \right)^{s+\frac{q-2l}{2}} H_{q-2l}  \left( Z_{i,n} \right)  
	\end{align*}
	To identify the leading order, let $r:=q-2l$, $w:=2s+q-2l=2s+r$ and $u := l-s$. Then $r=w-2s$, $l=s+u$ and $q=w+2u$. So by defining
	\begin{align*}
		D_r^{(w,m,p)}  =   \sum_{s=0}^{\lfloor p/2 \rfloor}  \mathds{1}_{A_{s,r}} (w) \sum_{u=0}^{\lfloor p/2 \rfloor}  \mathds{1}_{B_{m,p}}(w+2u) \,  c_{w+2u} \,  C_{w+2u,s+u} \,  \binom{s+u}{s} (-1)^{u} ,
	\end{align*}
	by setting $A_{s,r}= \{ x : x=2s+r \}$ and $B_{m,p} = \{ y : m \leq y \leq p \}$, we have
	\begin{align*}
		f\left(X_{i,n}\right) &= \sum_{w=0}^p \left(  \mathrm{Var} X_1^{(n) } \right)^{w/2} \sum_{r=0}^w D_j^{(w,m,p)}  H_{r}  \left( Z_{i,n} \right) .
	\end{align*}
	Now observe that $H_0 \big(Z_r^{(n)}\big)=1$, leading to
	\begin{align*}
		f\left(X_{i,n}\right) - \E f\left(X_{i,n}\right) &= \sum_{w=1}^p \left(  \mathrm{Var} X_1^{(n) } \right)^{w/2} \sum_{r=1}^w D_j^{(w,m,p)}  H_{r}  \left( Z_{i,n} \right) .
	\end{align*}
	Finally, define the leading power index as $  m_* := \min \{w \geq 1:  \sum_{r=1}^p D_r^{(w,m,p)}  \neq 0 \}$. The statement then follows by the orthogonality of Hermite polynomials and  \eqref{eq:2_Lambda}, as
	$$  f\big(X_{i,n}\big) - \E f\big(X_{i,n}\big) = \big( \mathrm{Var} X^{(n)} \big)^{m_*/2} \sum_{r=1}^{m_*} D_r^{(m_*,m,p)}   H_{r}(Z_{i,n}) + o_n \left( \big( \mathrm{Var} X^{(n)} \big)^{m_*/2} \right).
	$$   
\end{proof}

\begin{Remark}\normalfont
	\label{remark:2.2}
	Lemma \ref{lemma: variance} identifies those components which drive the limiting distribution. In cases (i) and (iii), they are expressed in the standardized variables $Z_{j,n}$, $j=1,\dots, n$. Note that the standardized variables satisfy recursion \eqref{Z_intro} with $\sigma_n^2=1-\alpha_n^2$. To unify notation, we also restrict attention to $\sigma_n^2=1-\alpha_n^2$ (i.e.~$\sigma^2=1$) in case (ii) as $\sigma^2$ is irrelevant for the mathematical treatment there.
\end{Remark}
\begin{Remark}\normalfont
	\label{remark:2.3}
	One can also
	derive the asymptotic covariances, noting that 
	\begin{align*}
		\mathrm{Cov} \left( \hat{S}_{n,t}(f) , \hat{S}_{n,s}(f) \right) = \frac{ (1 + o_n(1))}{2 \mathrm{Var} S_{n,1}(f) } \E \left[ S^2_{n,t}(f) + S^2_{n,s}(f) - \left( S_{n,t}(f) - S_{n,s}(f) \right)^2 \right]
	\end{align*}
	and $\E \big[ S_{n,t}(f) - S_{n,s}(f) \big]^2 = \E \big[ S_{n,t-s}(f) \big]^2 (1 + o_n(1))$ due to stationarity.
\end{Remark}
Finally, we study the additive functional $\hat{S}_n(f_n))$ along the sequence of polynomials $f_n=\sum_{q=m}^p c_q \delta_n^q H_q(\cdot/\delta_n)$ with $\delta_n^2=\mathrm{Var} X_{1,n}=\sigma_n^2/(1-\alpha_n^2)$.  
\begin{Lemma}
	\label{lemma:2.4}
	Consider $f=f_n=\sum_{q=m}^p c_q  H_{q,n} =\sum_{q=m}^p c_q  H_q(\cdot;\mathrm{Var} (\cdot))$  for some fixed family $(c_q)$. Then
	\begin{itemize}
		\item[(i)] If $\sigma_n^2 \gg  (1 - \alpha_n^2) $, then $  \hat{S}_{n,t}(f_n) = c_p \hat{S}_{n,t}(H_p^{(n)})  + R_{n,t}(f_n)$,
		where $R_{n,t}(f_n) \longrightarrow_{L^2} 0$ for $n \to \infty$. Moreover, for any $t\in [0,1]$
		\begin{align*}
			\hat{S}_{n,t}(H_p^{(n)}) &= \frac{ \sum_{i=1}^{\lfloor nt \rfloor}   \big( \mathrm{Var} X_{1,n} \big)^{p/2}  H_{p}  \big( Z_{i,n} \big)}{\sqrt{\mathrm{Var}S_{n,1}(H_p^{(n)})}}  \quad\text{and}\\
			\mathrm{Var} S_{n,t}(H_p^{(n)}) &=  \left( \mathrm{Var} X_{1,n} \right)^{p}  p! \, \Lambda_{p}(n,t). 
		\end{align*}
		\item[(ii)] If $\sigma_n^2 = \sigma^2 (1 - \alpha_n^2) $, then for any $t\in [0,1]$
		\begin{align*}
			\mathrm{Var} S_{n,t}(f_n) = \mathrm{Var} S_{n,t}(f) = \sum_{q=m}^p  c_q^2 q!  \left(   \sigma^2  \right)^q \Lambda_{q}(n,t) .
		\end{align*}  
		\item[(iii)] If $\sigma_n^2 \ll  (1 - \alpha_n^2) $, then 
		$  \hat{S}_{n,t}(f_n) = c_m \hat{S}_{n,t}(H_{m,n})  + R_{n,t}(f_n)$,
		where $R_{n,t}(f_n) \longrightarrow_{L^2} 0$ for $n \to \infty$. Moreover, for any $t\in [0,1]$
		\begin{align*}
			\hat{S}_{n,t}(H_{m,n}) &= \frac{ \sum_{i=1}^{\lfloor nt \rfloor}   \big( \mathrm{Var} X_{1,n} \big)^{m/2}  H_{m}  \big( Z_{i,n} \big)}{\sqrt{\mathrm{Var}S_{n,1}(H_{m,n})}}  \quad\text{and}\\
			\mathrm{Var} S_{n,t}(H_{m,n}) &=  \left( \mathrm{Var} X_{1,n} \right)^{m}  m! \, \Lambda_{m}(n,t). 
		\end{align*}
	\end{itemize}
\end{Lemma}
\begin{proof}
	The proof is analogous to the one of case (ii) in Lemma \ref{lemma: variance},  using the expression $  H_q \left( x ;  \sigma^2 \right)   =\sigma^{q} H_q (x/\sigma) $ and exploiting the orthogonality of Hermite polynomials.
\end{proof}

\subsection{Preliminaries on Gaussian Hilbert spaces}

For an exhaustive reference on Wiener chaos and multiple Wiener-It\^o integrals, we refer to \cite{Kuo2006}, \cite{PeccatiTaqqu2011}, \cite{Nourdin2012}, \cite{Major2014},  and \cite{Kallenberg2021}. Here, only notation is provided which is required  for our subsequent proofs. For $n \in \N$ fixed, consider  the sequence of $Z_{1,n}, Z_{2,n},\dots$   and the associated real separable Hilbert space 
\begin{align*}
	H^{(n)} := \overline{\text{span}\left\{ Z_{1,n}, Z_{2,n}, \ldots \right\} }^{\mathcal{L}^2 (\Omega )}.
\end{align*}
For any $n\in\N$, this space is isometrically isomorphic to $\mathcal{L}^2 (\R^+ )$. We set $e_{k,n}= \Phi ( Z_{k,n} )$, where $\Phi: H \to \mathcal{L}^2 (\R^+ )$ is an isometry. The related autocovariance kernel $ ( e_{k,n} )_{k \in \N}$ satisfies 
$ 
\E \big[ Z_{k,n} Z_{l,n} \big] {=}  \int_{0}^\infty e_{k,n} (x) e_{l,n} (x) \mathrm{d} x$, $ k,l\in \N$,
and the time domain representation
\begin{align*}
	\big( Z_{j,n} \big)_{j \in \N} \overset{\mathcal{D}}{=} \Big( \int_{0}^\infty e_{j,n} (t)  \mathrm{d} B (t) \Big)_{j \in \N}
\end{align*}
with a standard Brownian motion $B= \left( B(t) \right)_{t \geq0}$.
The additive functional $\hat{S}_{n,t}(H_q)$ built from $Z_{1,n},\dots,  Z_{n,n}$ is known to belong to the $q$-th Wiener chaos, as we can write $$	\hat{S}_{n,t}(H_q) =  I_q^B (g_{ \lfloor nt \rfloor}),$$ where $I_q^B$ is a multiple Wiener-It\^o Integral of order $q$ in  $B$,
\begin{align}
	(g_{ \lfloor nt \rfloor}) := \frac{\sum_{i=1}^{\lfloor nt \rfloor} \big( e_{i,n} \big)^{\otimes q}}{\sqrt{\mathrm{Var} S_{n,1}(H_q)  }} ,\label{eq:g_n}
\end{align}
and $( e_{i,n} )^{\otimes q}$ the $q$-fold  tensor product. Finally, $g_1 \otimes_r g_2\in L^2(\R_+^{2q-2r})$ denotes the r-th contraction of two functions $g_1, g_2 \in L^2(\R^q_+)$, with $g_1\otimes_0 g_2:=g_1\otimes g_2$ being the tensor product, and with $g_1 \tilde{\otimes}_r g_2$ we denote its symmetrization.

\subsection{Tightness in the Skorokhod space $D[0,1])$}
For later purposes, we prove tightness of $ \big(  S_{n,t}(f) \big)_{t \in [0,1]}$ for a polynomial $f$. In view of Remark \ref{remark:2.2}, we restrict attention to the case $\sigma_n^2=1-\alpha_n^2$. 

\begin{Lemma}
	\label{lemma: tightness}
	Let $\sigma_n^2=1-\alpha_n^2$ and $f=\sum_{q=m}^p c_q H_q$. Then for any $\beta>0$, $\gamma \in (0,1)$ and $\sigma_n^2 >0$ the sequence $ \big( \hat{S}_{n,t}(f) \big)_{t \in [0,1]}$ is tight. 
\end{Lemma}

\begin{proof}
	Write $\hat{S}_{n,t} (f) =\sum_{q=m}^p w_{n,q} \hat{S}_{n,t}(H_q)$ with  
	\begin{align}
		w_{n,q} := c_q \sqrt{\operatorname{Var} S_{n,1}(H_q) / \operatorname{Var} S_{n,1}(f)}. \label{eq:w_{n,q}}
	\end{align}
	 Let $0 \leq s \leq t \leq u \leq 1$ and $\Delta(t,s) := \hat{S}_{n,t}(f)  - \hat{S}_{n,s}(f)$. Then 
	\begin{align*}
		\left( \E \left[ \left|  \Delta(u,t) \right|^r  \left| \Delta(t,s) \right|^r \right] \right)^{1/r} \leq \left( \E  \left| \Delta(u,t)  \right|^{2r} \right)^{1/(2r)} \left( \E  \left|  \Delta(t,s)  \right|^{2r} \right)^{1/(2r)}
	\end{align*}
	by the Cauchy-Schwarz inequality. Let $v=2r$. By the Nelson hypercontractivity (Chapter~4 in \cite{Nelson1973} or Theorem~4.5 in \cite{Nourdin2012}) on a fixed Wiener chaos, 
	\begin{align*}
		\left( \E  \left|  \hat{S}_{n,t}(H_q) - \hat{S}_{n,s}(H_q)  \right|^v \right)^{1/v} \leq (v-1)^{q/2}  \| \hat{S}_{n,t}(H_q) - \hat{S}_{n,s}(H_q) \|_2  .
	\end{align*}
	Now applying Minkowski's inequality, we arrive at
	\begin{align*}
		\Big( \E \left[ | \Delta(s,t)|^v \right] \Big)^{1/v} &\leq \sum_{q=m}^p |w_{n,q}| \Big( \E \big| \hat{S}_{n,t}(H_q) - \hat{S}_{n,s}(H_q)\big|^v  \Big)^{1/v} \\
		&\leq \sum_{q=m}^p |w_{n,q}| (v-1)^{q/2} \| \hat{S}_{n,t}(H_q) - \hat{S}_{n,s}(H_q) \|_2.
	\end{align*}
	In view of Lemma \ref{lemma: variance} and $e^{- \gamma q (t-s) } \leq 1$, $e^{- \gamma q} - 1 + \gamma q \geq \gamma q - 1$  for each $\gamma \geq 0 $ (for $\beta=1$), the statement follows by \cite{Billingsley1999} equation (13.14), choosing $v/2=r>2$.
\end{proof}

\section{Brownian motion as limit process in the regime {\boldmath$n \gg t_{\text{mix},n} $}}
\label{sec: beta<1}

Recall that $n \gg t_{\text{mix},n} $ if and only if $\beta <1$.

\begin{proof}[Proof of Theorem \ref{prop: QLT} (i)]
For $q=1$ we have $ \hat{S}_{n,1}(H_1), B(1) \sim \mathcal{N}(0,1)$ for each $n$ and this implies $d_{TV}(\mathcal{L}\big(\hat{S}_{n,1}(H_1)\big) , \mathcal{L}(B(1)))=0$. For $q \geq 2$, $\hat{S}_{n,t}(H_q)$ belongs to the q-th Wiener chaos and therefore enables an application of the quantitative version of the \emph{Fourth-Moment Theorem} (cf. \cite{NualartPeccati2005}, \cite{NourdinPeccati2009} and \cite{NourdinPeccatiReinert2010}). To evaluate the contractions, we first show two Lemmas.

\begin{Lemma}
	\label{lemma:3_convexity}
	Let $q \in \N \backslash \{1 \}$ and ${r \in \left\{ 1 , \ldots , q-1 \right\}}$. Define
	\begin{align*}
		s(n,r):= \sum_{i,j,k,l=1}^n     \alpha_n^{r|i-j|}   \alpha_n^{r|k-l|}   \alpha_n^{(q-r)|i-k|}   \alpha_n^{(q-r)|j-l|}  .
	\end{align*}
	Then it holds $\displaystyle \max_{r \in \left\{ 1 , \ldots , q-1 \right\}} s(n,r) = s(n,1) = s(n,q-1) $.
\end{Lemma}
\begin{proof}[Proof of Lemma \ref{lemma:3_convexity}]
	The identity $s(n,1) = s(n,q-1)$ is immediate by symmetry. Next, with the abbreviations
	$$w_{i,j,k,l} := \alpha_n^{q(|i-k|+|j-l|)}  \in (0,1) \quad \text{ and } \quad  u_{i,j,k,l} := (|i-j|+|k-l| - |i-k| - |j-l|) \in \mathbb{Z},$$ 
	we can rewrite $ s(n,r) = \sum_{i,j,k,l=1}^n    w_{i,j,k,l}  e^{r  u_{i,j,k,l} \ln{\alpha_n}} $, which is easily  verified to satisfy the inequality 
	$$s''(n,r)\cdot s(n,r)\geq \big(s'(n,r)\big)^2$$ 
	and is therefore log-convex in $r$. That is, $  s^2(n,r) \leq s(n,r-1) s(n,r+1) $.
	Now, we introduce the ratios $\phi_r := s(n,r)/s(n,r-1)$, and remark that, for $r \in \{1,\ldots,q-1\}$,
	\begin{align*}
		\frac{s(n,r)}{s(n,1)} = \prod_{k=2}^r \phi_k  \leq \left( \prod_{k=2}^{q-1} \phi_k \right)^{\frac{r-1}{q-2}}= \left( \frac{s(n,q-1)}{s(n,1)} \right)^{\frac{r-1}{q-2}} = 1 .
	\end{align*}
	The inequality is deduced as follows: 
	with $y_k:=\mathds{1}_{[2,r]}(k)$, $(y_k)$ is  non-negative and  non-increasing in $k$.  Isotonicity of $(\phi_r)_r$ is implied by log-convexity. Then Chebychev's sum inequality (cf.~\cite{HardyLittlewoodPolya}) reveals 
	\begin{align*}
		\frac{1}{q-2} \sum_{k=2}^{q-1} (\log \phi_k )y_k \leq \left( \frac{1}{q-2} \sum_{k=2}^{q-1} \log \phi_k \right) \frac{r-1}{q-2} .
	\end{align*} 
	Applying the exponential function on both sides provides the inequality. 
\end{proof}
\begin{Lemma}
	\label{lemma:3_order}
	Let $q \in \N \backslash \{1 \}$. Then we have $s(n,1) = C(\gamma,q) D(n) (1+o_n(1))$ with   $   D(n) \asymp n^{3\beta + 1} $ and $C(\gamma,q) \leq 2 / ( (q-1) \gamma^3)$.
\end{Lemma}
\begin{proof}[Proof of Lemma \ref{lemma:3_order}]
	The proof of this expression relies in identifying an upper and a lower bounds for $ s(n,1) $. For the upper bound we observe $\sum_{j=1}^n \alpha_n^{|i-j|} =   ( 1  + o_n(1) ) \, n^\beta / \gamma$. This leads to $$ \sum_{j,l=1}^n     \alpha_n^{|i-j|}  \alpha_n^{|k-l|} \alpha_n^{(q-1)|j-l|} \leq ( 1  + o_n(1) ) \, \frac{n^{2\beta }}{\gamma^2},$$ and from \eqref{eq:2_Lambda} $s(n,1) \leq ( 1  + o_n(1) ) ( 2 n^{3 \beta +1} ) / ( ( q-1) \gamma^3)$ follows. For the lower bound, we estimate $$ s(n,1) \geq \sum_{i,j,k,l=1}^n   \alpha_n^{2(q-1)|i-j|}   \alpha_n^{2(q-1)|k-l|}   \alpha_n^{2(q-1)|i-k|}. $$ Then define $ S_i := \sum_{j=1}^n \alpha_n^{2(q-1)|i-j|}$.
	We observe that the sequence $(S_i)_{1 \leq i \leq n}$ is symmetric in $i$ and $ \min_{1 \leq i \leq n} S_i = S_1 = S_n$. The claimed lower bound follows by analogous arguments to the ones mentioned for the upper bound.
\end{proof}
  We are now ready to prove the quantitative bound of the statement. As $(Z_{1,n}, Z_{2,n}, \ldots)$ is a stationary sequence of standard Gaussian random variables, the additive functional $S_{n,t}(H_q)$ belongs to the q-th Wiener chaos and can be therefore represented as a multiple Wiener-It\^o integral.  Observe that $g_n$ in \eqref{eq:g_n} belongs to $\in L^2( \R^q )$ and is symmetric. Theorem 3.1 in \cite{NourdinPeccatiReinert2010} yields
\begin{align*}
	d_{TV} \left( \mathcal{L}\big(\hat{S}_{n,1}(H_q)\big) , \mathcal{L}(B(1)) \right) \leq 2 \sqrt{q^2 \sum_{r=1}^{q-1} (r-1)!^2 \binom{q-1}{r-1}^4 (2q-2r)!   \| g_n \tilde{\otimes}_r g_n \|^2_{L^2\left( \R^{2q-2r} \right)} } , 
\end{align*}
where $\| g_n \tilde{\otimes}_r g_n \|^2_{L^2\left( \R^{2q-2r} \right)} \leq \| g_n \otimes_r g_n \|^2_{L^2\left( \R^{2q-2r} \right)} $ (cf. Lemma 14.22 in \cite{Kallenberg2021}).

\noindent
Using $\E \big[Z_{j,n} Z_{k,n} \big]= \alpha_n^{|j-k|}$,  one obtains as in the proof of Theorem 1.1, \cite{Nourdin2012} on page 7, 
\begin{align*}
	\| g_n \otimes_r g_n \|^2_{L^2\left( \R^{2q-2r} \right)}     &=  \frac{1}{\left( \mathrm{Var} S_{n,1} (H_q) \right)^2}   \sum_{i,j,k,l=1}^n   \alpha_n^{r|i-j|}   \alpha_n^{r|k-l|}   \alpha_n^{(q-r)|i-k|}   \alpha_n^{(q-r)|j-l|}.
\end{align*}
Thus, with Lemmas \ref{lemma:3_convexity}, \ref{lemma:3_order}, \ref{lemma: variance}, and equation \eqref{eq:2_Lambda}, we get
\begin{align*}
	\| g_n \tilde{\otimes}_r g_n \|^2_{L^2\left( \R^{2q-2r} \right)}  \leq \left( \frac{q \gamma}{2 q! n^{\beta+1}} \right)^2 \frac{2 n^{3 \beta +1}}{(q-1) \gamma^3} (1+ o_n(1)) = \frac{q^2}{2 q!^2 (q-1) \gamma} n^{\beta-1} (1+ o_n(1)),
\end{align*}
and therefore
\begin{align*}
	d_{TV} \big( \mathcal{L}\big(\hat{S}_{n,1}(H_q)\big) , \mathcal{L}(B(1)) \big) \leq C(q) \gamma^{-1/2} n^{\frac{\beta-1}{2}} (1+o_n(1)),
\end{align*}
with $C(q)=  \frac{q^2}{q!} \sqrt{ \frac{2}{(q-1) }   \sum_{r=1}^{q-1} (r-1)!^2 \binom{q-1}{r-1}^4 (2q-2r)! }$.
\end{proof}

On basis of Theorem \ref{prop: QLT} (i), we now continue with the functional limit theorem before the phase transition, providing in particular the first row in Figure \ref{fig}.
\begin{Theorem}
	\label{theorem:beta<1}    
	Let $\beta<1$. 
	\begin{itemize}
		\item[(i)] If $\sigma_n^2 \gg  (1 - \alpha_n^2) $ consider $f=\sum_{q=m}^p c_q H_q$. Then 
		$$\left( \hat{S}_{n,t}(f) \right)_{t \in [0,1]} \longrightarrow_{\mathcal{D}}   \big(B(t) \big)_{t\in[0,1]} \qquad \text{in} \; D[0,1]  \text{ as }n\rightarrow\infty;$$
		\item[(ii)] If $\sigma_n^2 = 1 - \alpha_n^2 $ consider $f \in L^2 \left(\mathcal{N} \left(0, 1 \right) \right)$. Then 
		$$\left( \hat{S}_{n,t}(f) \right)_{t \in [0,1]} \longrightarrow_{\mathcal{D}}    \big(B(t) \big)_{t\in[0,1]} \qquad \text{in} \; D[0,1]  \text{ as }n\rightarrow\infty;$$
		\item[(iii)] If $\sigma_n^2 \ll  (1 - \alpha_n^2) $ consider $f=\sum_{q=m}^p c_q H_q$. Then 
		$$ \left( \hat{S}_{n,t}(f) \right)_{t \in [0,1]} \longrightarrow_{\mathcal{D}}   \big(B(t) \big)_{t\in[0,1]}  \qquad \text{in} \; D[0,1]  \text{ as }n\rightarrow\infty.$$
	\end{itemize}
\end{Theorem}

\begin{proof}
	In view of Remark \ref{remark:2.2}, we restrict attention to the recursion \eqref{Z_intro} with $\sigma_n^2=1-\alpha_n^2$.	 The proof proceeds in three steps. We first establish the limit distribution for $f=H_q$ exploiting Theorem \ref{prop: QLT} (i). Thereafter, we extend this result to linear combinations of Hermite polynomials for all three cases. Here, we make use of Proposition~1 and Theorem 1 in \cite{PeccatiTudor2005}, which represent multidimensional generalizations of the Fourth-Moment Theorem. Based on Lemma~\ref{lemma: tightness}, we only need to show convergence of finite-dimensional distributions (cf.~\cite{Billingsley1999}, Theorem 13.5). The extension to arbitrary $f \in L^2(\mathcal{N}(0,1))$ in (ii) is finally based on the sharp maximal inequality of \cite{PeligradUtevWu2007} for stationary sequences.

	\paragraph*{Step 1: $f=H_q$} Recall that $Z_{j,n}$, $j,\dots, n$, fulfills the recursion
	$$
	Z_{j,n}=\alpha_n Z_{j-1,n}+\sqrt{1-\alpha_n^2}\varepsilon_j,
	$$
	which is the scenario in (ii).
	As a consequence of Theorem \ref{prop: QLT} (i) (replacing in there $\gamma$ by $\gamma t^\beta$ and $n$ by $\lfloor nt \rfloor$), $\hat{S}_{n,t}(H_q) \longrightarrow_{\mathcal{D}} \mathcal{N} \left( 0, t \right)$  for any $t\in[0,1]$. To extend convergence in law from the one-dimensional marginals to the fidis $$\big( \hat{S}_{n,t_1}(H_q), \ldots ,\hat{S}_{n,t_d}(H_q) \big), \quad 0\leq t_1<\dots <t_d\leq 1,$$   Proposition 1 of \cite{PeccatiTudor2005} is employed.  From the covariance structure, we identify the limit as a standard Brownian motion.

	\paragraph*{Step 2: $f=\sum_{q=m}^p c_q H_q$} 
	We start with case (ii). First observe
	\begin{align*}
		w_{n,q} 
		\overset{\ref{lemma: variance} }{=}   \sqrt{\frac{2 c_q^2  q!  \frac{ n^{\beta+1}   }{q \gamma} \left( 1 + o_n(1) \right)}{  2 \sum_{r=m}^p  c_r^2 r!   \frac{ n^{\beta+1}   }{r \gamma} \left( 1 + o_n(1) \right)      } } 
		\longrightarrow \sqrt{ \frac{\frac{ c_q^2 q!  }{ q}}{\sum_{r=m}^p \frac{ c_r^2 r!  }{ r }} } 
		= : b_q
	\end{align*}
	as $n \to \infty$. Thus, together with Lemma \ref{lemma: tightness}, 
	\begin{align}
		\sum_{q=m}^p (w_{n,q} -  b_q ) \hat{S}_{n,t}(H_q) \longrightarrow_{\Pro} 0. \label{eq:3.2}
	\end{align}
	Now consider $0 \leq t_1 < \ldots < t_d \leq 1$ and $\kappa_1, \ldots, \kappa_d \in \R$ for $d \in \N$. 
	From the previous step,  $ \big(  \hat{S}_{n,t_1} (H_q), \ldots ,  \hat{S}_{n,t_d} (H_q) \big)   \rightarrow_{\mathcal{D}}    \mathcal{N} \left( 0, \bf{C}_d \right) $ as $n \to \infty$  for each $q \in \N$, with $C_{ii}=t_i$ and $C_{ij}=\min(t_i,t_j)$. Then $ \kappa_1 \hat{S}_{n,t_1}(H_q) + \ldots +   \kappa_d \hat{S}_{n,t_d} (H_q)  \longrightarrow_{\mathcal{D}}  \kappa_1 B(t_1) + \ldots +  \kappa_1 B(t_d) $ as $n \to \infty$ by the continuous-mapping theorem. The next goal is to extend this convergence to different chaos orders. 
	Considering that
	\begin{align*}
		\mathrm{Var} \left( \sum_{j=1}^{d} \kappa_{j} \hat{S}_{n,t_{j}}(H_q)  \right) =  \sum_{j=1}^{d} \kappa_{j}^2 t_{j} + 2 \sum_{1\leq i < j \leq {d}} \kappa_{i} \kappa_{j} \min(t_{i},t_{j}) =: V
	\end{align*}
	and $ \kappa_{1} \hat{S}_{n,t_{1}}(H_q)  + \ldots + \kappa_{d} \hat{S}_{n,t_{d}}(H_q)$ remains a multiple Wiener-It\^{o} integral of order~$q$, 
	\begin{align*}
		\left( \sum_{j=1}^{d} \kappa_{j} \hat{S}_{n,t_{j}}(H_m)  , \ldots , \sum_{j=1}^{d} \kappa_{j} \hat{S}_{n,t_{j}}(H_p) \right)  \longrightarrow_{\mathcal{D}}  \mathcal{N} \big( 0,  V \cdot \mathrm{I}_{ p-m +1}  \big)
	\end{align*}
	as $n \to \infty$ by  Theorem 1 of \cite{PeccatiTudor2005}, where $\mathrm{I}_k$ denotes the $k \times k$ identity matrix. By the continuous-mapping theorem again, the latter implies that $$\sum_{j=1}^d \kappa_{j}  \sum_{q=m}^p  b_q \hat{S}_{n,t_j} (H_q)  \longrightarrow_{\mathcal{D}}    \sum_{j=1}^d \kappa_{j} \sum_{q=m}^p  b_q  B_q (t_j)$$ as $n \to \infty$, with independent Brownian motions $B_q$. Again by  the Cramér-Wold  device,
	\begin{align*}
		\Bigg(  \,  \sum_{q=m}^p  b_q \hat{S}_{n,t_1} (H_q)  , \ldots , \sum_{q=m}^p b_q \hat{S}_{n,t_d} (H_q) \Bigg)  \longrightarrow_{\mathcal{D}}    \Bigg(  \sum_{q=m}^p b_q B_q (t_1)  , \ldots , \sum_{q=m}^p b_q B_q (t_d) \Bigg) 
	\end{align*}
	as $n \to \infty$. Together with \eqref{eq:3.2} and Lemma \ref{lemma: tightness},
	\begin{align*}
		\left(   \hat{S}_{n,t} (f)  \right)_{t \in [0,1]} &=  \left(\, \sum_{q=m}^p w_{n,q} \hat{S}_{n,t} (H_q)  \right)_{t \in [0,1]} \longrightarrow_{\mathcal{D}} \left(  \, \sum_{q=m}^p b_q B_q(t) \right)_{t \in [0,1]}  \overset{\mathcal{D}}{=} \big( B(t) \big)_{t \in [0,1]} 
	\end{align*}
	in $D[0,1]$ as $n \to \infty$, as $\sum_{q=m}^p b_q^2 = 1$. 
	
	\smallskip
	
	For cases (i) and (iii), we make first use of Lemma \ref{lemma: variance} to extract the leading components. As these are polynomials in $Z_{j,n}$, the proof is analogous to the one for (ii).
	
	\paragraph*{Step 3: $f \in L^2 \left( \mathcal{N} (0, 1) \right)$}
	First note that as a consequence of Step 2, there exists a sequence $(p_n)_{n \in \N}$ with $p_n \nearrow \infty$ such that
	$$
	d_{BL} \Big( \mathcal{L} \big( \hat{S}_{n, \cdot} (f^{(p_n)} ) \big) , \mathcal{L} (B) \Big) \longrightarrow 0  ,
	$$
	where $d_{BL}$ denotes the dual bounded Lipschitz metric, metrizing weak convergence. Therefore it is sufficient to prove 
	$
	\sup_{t \in [0,1]} \big| \hat{S}_{n,t}\big(f - f^{(p_n)}\big) \big| \longrightarrow_\Pro 0$.
	To this aim, let $$T_{i,n} := \sum_{q=p_n+1}^\infty c_q H_q ( Z_{i,n} ); \qquad S_{j,n} := \sum_{i=1}^j T_{i,n}; \qquad S_l^{* (n)} = \max_{1 \leq j \leq l} \big| S_{j,n} \big|.$$ Observe that $T_{i,n}$ is stationary in $i$ for fixed $n \in \N$ and $\E \big(T_{i,n}\big)^2 = \sum_{q=p_n+1}^\infty c_q^2 q! < \infty$ by Parseval's identity. Consider the filtration 
	$$\big(\mathcal{F}_i\big)_{i \in \{0,\ldots,n\}} \text{ with } \mathcal{F}_i:= \sigma \big( Z_{0,n}, \varepsilon_k : k \leq i \big). 
	$$ 
	Clearly $T_{i,n} \in \mathcal{F}_i$ for each $n \in \N$. Then Corollary 1 of \cite{PeligradUtevWu2007} is applicable with $p=2$ (cf. also its first version in \cite{PeligradUtev2005}), and we obtain together with Lemma \ref{lemma:5.2} that $\| S_l^{* (n)} \|_2$ is bounded by
	\begin{align*}
		&C_2^{1/2} l^{1/2} \left[ \| T_{i,n} \|_2 + 240 \sum_{k=1}^l k^{-1/2} \| \E \big[ T_{k,n} | \mathcal{F}_{0,n} \big] \|_2 \right] \\
		&\, = C_2^{1/2} l^{1/2} \left[   \sqrt{ \hspace{-0.5mm} \sum_{q=p_n+1}^\infty \hspace{-2mm} c_q^2 q!}  + 240 \sum_{k=1}^l k^{-1/2} \bigg\Arrowvert \hspace{-0.5mm}  \sum_{q=p_n+1}^\infty \hspace{-2mm} c_q  \E \big[ \alpha_n^{qk} H_q (Z_{0,n}) + R_q(Z_{k,n})| \mathcal{F}_{0,n} \big] \bigg\Arrowvert_2 \right] \\
		&\, \leq C_2^{1/2} l^{1/2} \left[ \sqrt{\sum_{q=p_n+1}^\infty c_q^2 q!} + 240 \,  \sqrt{\sum_{q=p_n+1}^\infty c_q^2 q!} \sum_{k=1}^l k^{-1/2} \alpha_n^{k} \right],
	\end{align*}
	where we have used in the last inequality that $\E \big[ R_q(Z_{k,n})| \mathcal{F}_{0,n} \big]=0$. Hence, with 
	$$
	\sup_{t \in [0,1]} \left|  \hat{S}_{n,t}\big(f - f^{(p_n)}\big) \right| = \frac{S_n^{* (n)}}{\sqrt{\mathrm{Var} S_{n,1}(f)}},
	$$
	Chebychev's inequality reveals
	$$
	\Pro \bigg(  \sup_{t \in [0,1]} \left|  \hat{S}_{n,t}\big(f - f^{(p_n)}\big) \right| > \lambda \bigg) \leq \frac{C_2 \, n }{\mathrm{Var} S_{n,1}(f) \,  \lambda^2} \sum_{q=p_n+1}^\infty c_q^2 q! \left[ 1 + 240 \,   \sum_{k=1}^n k^{-1/2} \alpha_n^{k} \right]^2
	$$
	for every $\lambda >0$, which converges to $0$ as $n \to \infty$ by the subsequent Lemma \ref{lemma:3.4}.
\end{proof}
\begin{Lemma}
	\label{lemma:3.4}
	For any $\beta > 0$, 
	$$
	\sum_{k=1}^n \frac{1}{\sqrt{k}} \Big( 1 - \frac{\gamma}{n^\beta} \Big)^{k} \leq \begin{cases}
		1 + \sqrt{\frac{n^{ \beta} \pi}{\gamma}} & \text{ if } \beta <1, \\
		1 + 2 \sqrt{n} & \text{ if } \beta \geq 1.
	\end{cases}
	$$
\end{Lemma}
\begin{proof}
	As $ k^{-1/2} \big( 1 - \gamma / n^\beta \big)^{k} $ is monotone decreasing in $k$, the integral criterium leads to 
	\begin{align}
		\sum_{k=1}^n \frac{1}{\sqrt{k}}  \Big( 1 - \frac{\gamma}{n^\beta} \Big)^{k} \leq 1 - \frac{\gamma}{n^\beta} + \int_1^n \frac{1}{\sqrt{x}}  \Big( 1 - \frac{\gamma}{n^\beta} \Big)^x \mathrm{d}  x . \label{eq:3.3}
	\end{align}
	If $\beta <1$, observe that $1-y \leq e^{-y}$ for $y \in [0,1]$ and substituting $z:= x \, (\gamma/n^\beta)$ leads to 
	$$
	\sum_{k=1}^n \frac{1}{\sqrt{k}}  \Big( 1 - \frac{\gamma}{n^\beta} \Big)^{k} \leq 1 + \sqrt{\frac{n^{\beta}}{\gamma}}  \int_{\gamma/n^\beta}^{\gamma n^{1-\beta}} z^{-1/2} e^{-z} \mathrm{d}  z.
	$$
	Extending the integration boundaries and recalling that $ \int_{0}^{\infty} z^{-1/2} e^{-z} \mathrm{d}  z = \Gamma (1/2) = \sqrt{\pi}$ completes the proof for $\beta \in (0,1)$. For $\beta \geq 1$, the right-hand side in \eqref{eq:3.3} is bounded by $1+2 \sqrt{n}$.
\end{proof}

\begin{Theorem}
	\label{theorem:beta<1_fn}    
	Let $\beta<1$ and $f_n=\sum_{q=m}^p c_q  H_q(\cdot;\mathrm{Var} (\cdot))$. Then
	$$\left( \hat{S}_{n,t}(f_n) \right)_{t \in [0,1]} \longrightarrow_{\mathcal{D}}   \big(B(t) \big)_{t\in[0,1]} \hspace{1mm} \text{in} \; D[0,1]  \text{ as }n\rightarrow\infty.$$
\end{Theorem}
\begin{proof}
	In view of Lemma \ref{lemma:2.4}, the proof is analogous to the one of Theorem~\ref{theorem:beta<1} and hence omitted.
\end{proof}

\section{Limit processes at the phase transition  {\boldmath$n\asymp t_{\text{mix},n}$}}
\label{sec: beta=1}

Recall that $n \asymp t_{\text{mix},n} $ if and only if $\beta =1$. 

\begin{Remark}\normalfont
	\label{remark:int}
	Note that 
	\begin{align*}
		\int_{\R^q} \left| \frac{e^{i \left( \sum_{r=1}^q \lambda_r \right) t} -1}{i \left( \sum_{r=1}^q \lambda_r \right)}  \prod_{r=1}^{q} \frac{1}{i \lambda_r + \gamma} \right|^2 \mathrm{d}  \lambda_1  \ldots \mathrm{d}  \lambda_q  \leq t^2 \left( \frac{\pi}{\gamma} \right)^q < \infty \qquad\text{for all $t \in [0,1]$}
	\end{align*}
	and therefore, by the isometry property of multiple Wiener-It\^{o} integrals, the limit process in Figure \ref{fig} and the subsequent statements is well-defined.
\end{Remark}

\begin{proof}[Proof of Theorem \ref{prop: QLT} (ii)]
	We exploit the spectral representation of $Z_{1,n}, \ldots , Z_{n,n}$ together with the expression of $S_{n,t}(H_q)$  as a multiple Wiener-It\^{o} integral. Specifically, 
\begin{align}
	H_q \left( Z_{j,n} \right)  = (\sigma_n^2)^{q/2} \int_{[-\pi,\pi]^q}^{''}  e^{i j \left( \sum_{r=1}^q \lambda_r \right) }  \prod_{r=1}^{q}  \phi^{(n)}(\lambda_r) \,  \mathrm{d} \tilde{B}_{0} (\lambda_1) \ldots \mathrm{d} \tilde{B}_{0}(\lambda_q)    \label{H_m}
\end{align}
for  $j =1, \ldots , n$, where $\tilde{B}_{0}(-\lambda)=\overline{\tilde{B}_{0}(-\lambda)}$, $\tilde{B}_{0}$ restricted to $\R_+$ is some complex Brownian motion, and
\begin{align}
	\phi^{(n)}(\lambda) :=  \frac{1}{\sqrt{2 \pi}}  \left(  \sum_{k=0}^\infty  \left( 1 - \frac{\gamma}{n} \right)^k  e^{- i k \lambda}  \right).   \label{phi}
\end{align}
We now consider the one-dimensional marginals. Summing over $j$ reveals for any $t \in [0,1]$,
	\begin{align}
	S_{n,t}(H_q) &= (\sigma_n^2)^{q/2} \int_{[-\pi,\pi]^q}^{''} \hspace{-2mm} \frac{e^{i \left( \sum_{r=1}^q \lambda_r \right)  \left( \lfloor (n-1) t \rfloor +1 \right) } - 1}{e^{i \left( \sum_{r=1}^q \lambda_r \right)} -1} \prod_{r=1}^{q}  \phi^{(n)}(\lambda_r)  \, \mathrm{d} \tilde{B}_{0} (\lambda_1) \ldots \mathrm{d} \tilde{B}_{0}(\lambda_q) .\nonumber
\end{align}
Exploiting the self-similarity $n^{-1/2} \tilde{B}_{0}(n A) \overset{\mathcal{D}}{=}  \tilde{B}_{0}(A)$ for each set $A \in \mathcal{B}[- \pi, \pi ]$, we find
\begin{align*}
	\big( S_{n,t}(H_q) \big)_{t \in [0,1]} &\overset{{\mathcal{D}}}{=} \bigg( n^{1+\frac{q}{2}}(\sigma_n^2)^{q/2} \int_{[- n \pi, n \pi]^q}^{''} \psi_n(\lambda_1, \ldots , \lambda_q,t) \,  \mathrm{d} \tilde{B}_{0} \left( \lambda_1 \right) \ldots \mathrm{d} \tilde{B}_{0} \left( \lambda_q \right) \bigg)_{t \in [0,1]}
\end{align*}
with
	\begin{align}
	\psi_n(\lambda_1, \ldots , \lambda_q,t) :=  \frac{e^{i \left( \sum_{r=1}^q \lambda_r/n \right)  \left( \lfloor (n-1) t \rfloor +1 \right) } - 1}{n \left( e^{i \left( \sum_{r=1}^q \lambda_r/n \right)} -1 \right)} \frac{1}{n^{q}} \prod_{r=1}^{q}  \phi^{(n)}   \left( \frac{\lambda_r}{n} \right) .  \label{psi_n}
\end{align}
Dividing by ${\sqrt{\mathrm{Var} {S}_{n,1}(H_q)}} $ yields
\begin{align*}
	\hat{S}_{n,1}(H_q) &\overset{{\mathcal{D}}}{=} K_{q,n}^{(\gamma)} \int_{[- n \pi, n \pi]^q}^{''} \psi_n(\lambda_1, \ldots , \lambda_q,1)  \, \mathrm{d} \tilde{B}_{0} \left( \lambda_1 \right) \ldots \mathrm{d} \tilde{B}_{0} \left( \lambda_q \right),
\end{align*}
where further reducing the fraction and letting  $n \to \infty$ reveals
\begin{align}
	K_{q,n}^{(\gamma)} := \frac{n^{1+\frac{q}{2}}(\sigma_n^2)^{q/2}}{\sqrt{\mathrm{Var} {S}_{n,1}(H_q)}}  \longrightarrow \sqrt{\frac{ 2^{q-1} q^2 \gamma^{q+2}}{  q! \,  \left(\exp(-\gamma q) - 1 + \gamma q \right)}} =: K_q^{(\gamma)}  \label{eq:K_q}
\end{align}
and more precisely $ \big(K_{q,n}^{(\gamma)} - K_{q}^{(\gamma)} \big)^2 \leq C_{\gamma,q}^2 n^{-2}$ for a constant $C_{\gamma,q}>0$.
Next, we evaluate  $\E \big( \hat{S}_{n,1}(H_q)	- W_{n,q}(1) \big)^2$, where
\begin{align}
	W_{n,q}(1) = K_{q,n}^{(\gamma)} \int_{[- n \pi, n \pi]^q}^{''} \psi(\lambda_1, \ldots , \lambda_q,1) \, \mathrm{d} \tilde{B}_{0} \left( \lambda_1 \right) \ldots \mathrm{d} \tilde{B}_{0} \left( \lambda_q \right)  
	\label{W_{m,n}}
\end{align}
with
\begin{align}
	\psi(\lambda_1, \ldots , \lambda_q,1) :=  \begin{cases}
		\frac{e^{i \left( \sum_{r=1}^q \lambda_r \right) } - 1}{i \sum_{r=1}^q \lambda_r} \frac{1}{\left(2 \pi \right)^{q/2}} \prod_{r=1}^{q}   \frac{1}{i \lambda_r + \gamma}  & \text{ if } \sum_{r=1}^q \lambda_r \neq 0;\\
		  \, \frac{1}{\left(2 \pi \right)^{q/2}} \prod_{r=1}^{q}   \frac{1}{i \lambda_r + \gamma}  & \text{ if } \sum_{r=1}^q \lambda_r = 0.
	\end{cases}  \label{psi}
\end{align}
Using the isometry for multiple Wiener-It\^{o} integrals (equation (5.5.62) in \cite{PeccatiTaqqu2011}), we obtain for arbitrary $M \in \R_{>0}$
$$\E \big( \hat{S}_{n,1}(H_q)	- W_{n,q}(1) \big)^2 = \big( K_{q,n}^{(\gamma)} \big)^2 \big( I_{M,n} + I_{M,n}^c \big),$$ where
\begin{align*}
	I_{M,n} &=  \int_{[- M, M]^q}  \left| \psi_n(\lambda_1, \ldots , \lambda_q,1) -  \psi (\lambda_1, \ldots , \lambda_q,1) \right|^2 \mathrm{d} \lambda_1 \ldots  \mathrm{d} \lambda_q \qquad\text{and}\\
	I_{M,n}^c &=   \int_{[- n \pi, n \pi]^q \setminus {[- M, M]^q} }  \left| \psi_n(\lambda_1, \ldots , \lambda_q,1) -  \psi (\lambda_1, \ldots , \lambda_q,1) \right|^2 \mathrm{d} \lambda_1 \ldots  \mathrm{d} \lambda_q .
\end{align*}
 Note that $\psi_n \to \psi$ pointwise in $\lambda_1, \ldots, \lambda_q$ as $n \to \infty$. As furthermore $ |\psi_n |, |\psi| \leq \gamma^{-q}$, which can be deduced by means of the inequality $|1 - c e^{-ix}| \geq 1 - c$, for $c \in [0,1]$, $I_{M,n}\rightarrow 0$ as $n\rightarrow\infty$, by the theorem of dominated convergence. More specifically, Taylor expansions in numerator and denominator of the components of $\psi_n$ and $\psi$ reveal the existence of a constant $C_{\gamma,q}>0$ such that
	\begin{align}
	I_{M,n} \leq \frac{1}{n^2} \frac{q^2}{2^{q+2}} \gamma^{2-q} + C_{\gamma,q} \bigg( \frac{\pi^q}{\gamma^q} \frac{(1+M^4)^2}{n^4} +  \frac{1+M^4}{n^3} \bigg) \label{eq:I_{M,n}}
\end{align}
as long as  $n \geq M$ and $n \geq 1 + M$.

\noindent
As concerns $I_{M,n}^c$, we first identify a bound on $\left| \psi_n \right|^2$. Note that
\begin{align*}
	\left|\frac{ e^{i \left( \sum_{r=1}^q \lambda_r/n \right)  n } - 1  }{  n \big( e^{i \left( \sum_{r=1}^q \lambda_r/n \right)} -1 \big) } \right| = \left| \sum_{j=0}^{ n-1  } \frac{  e^{i j \left( \sum_{r=1}^q \lambda_r/n \right) }}{n} \right| \leq 1 .
\end{align*}
For the second factor in \eqref{psi_n}, we obtain
\begin{align*}
	\left| \frac{1}{n^{ q}} \prod_{r=1}^{q}  \phi^{(n)}   \left( \frac{\lambda_r}{n} \right) \right|^2 &=   \frac{1}{n^{ 2q}} \prod_{r=1}^{q} \left| \frac{1}{\sqrt{2 \pi}}  \left(  \sum_{k=0}^\infty  \left[ \left( 1 - \frac{\gamma}{n} \right)  e^{- i  \lambda/n}  \right]^k \right) \right|^2 \\
	&= \frac{1}{n^{2 q}} \prod_{r=1}^{q}  \frac{1}{2 \pi}  \left(  \frac{1}{ 1 - 2(1-\gamma/n) \cos(\lambda/n) +  (1-\gamma/n)^2  } \right) \\
	&\leq\frac{1}{n^{2 q}} \prod_{r=1}^{q}  \frac{1}{2 \pi}  \left(  \frac{1}{ 1 - 2(1-\gamma/n) \left(1 - \frac{2}{\pi^2} \left( \frac{\lambda}{n} \right)^2 \right) +  (1-\gamma/n)^2  } \right) \\
	&= \frac{1}{n^{2 q}} \prod_{r=1}^{q}  \frac{1}{2 \pi}  \left(  \frac{1}{ (\gamma/n)^2+2(1-\gamma/n) \frac{2}{\pi^2}\big(\frac{\lambda}{n}\big)^2} \right)   ,    
\end{align*}
where we we have used the identity $1-\cos (x)=2\sin(x/2)^2$ together with
Jordan's inequality $\sin^2 (x/2) \geq   x^2 / \pi^2$ for $x \in [-\pi,\pi]$ in the  inequality. The last expression is bounded by
\begin{align}
	\frac{1}{(2\pi)^{q}}\prod_{r=1}^{q} \bigg(\gamma^2 + \frac{2}{\pi^2} \lambda_r^2 \bigg)^{-1} \label{eq:4.7}
\end{align}
for $n\geq 2\gamma$, which delivers the required uniform bound on $\arrowvert\psi_n\arrowvert^2$. Analogously, 
$$\arrowvert \psi\arrowvert^2 \leq  \frac{1}{(2\pi)^{q}}\prod_{r=1}^{q} \frac{1}{\gamma^2 + \lambda_r^2}.$$
Now observe that $\R^q \setminus [-M,M]^q = \bigcup_{j=1}^q \{ \lambda=(\lambda_1, \ldots, \lambda_q) \in \R^q : | \lambda_j | > M  \} = \bigcup_{j=1}^q A_j$. Therefore, with $| \psi_n(\lambda,1) - \psi(\lambda,1)|^2 \leq 2 |\psi_n(\lambda,1)|^2 + 2 |\psi(\lambda,1)|^2 $,
\begin{align}
	I_{M,n}^c 		&\leq \int_{\bigcup_{j=1}^q A_j}  \frac{2}{(2\pi)^{q}} \Bigg[ \prod_{r=1}^{q} \bigg(\gamma^2 + \frac{2}{\pi^2} \lambda_r^2 \bigg)^{-1} + \prod_{r=1}^{q} \frac{1}{\gamma^2 + \lambda_r^2} \Bigg] \mathrm{d} \lambda \notag \\
	&\leq \frac{2}{(2\pi)^{q}} \sum_{j=1}^q \left[ \int_{ A_j}  \prod_{r=1}^{q} \bigg(\gamma^2 +  \frac{2}{\pi^2} \lambda_r^2 \bigg)^{-1} \mathrm{d} \lambda + \int_{ A_j} \prod_{r=1}^{q} \frac{1}{\gamma^2 + \lambda_r^2}  \mathrm{d} \lambda \right]. \notag
\end{align}
For fixed $j$ Fubini's theorem and the inequality $(\gamma^2 + (2/\pi^2)x^2)^{-1} \leq ((2/\pi^2)x^2)^{-1}$ provide
\begin{align*}
	\int_{ A_j}  \prod_{r=1}^{q} \bigg(\gamma^2 +  \frac{2}{\pi^2} \lambda_r^2 \bigg)^{-1} \mathrm{d} \lambda 	\leq \frac{\pi^2}{2} \left( \int_{|x|>M} \frac{1}{x^2} \mathrm{d} x \right) \left( \frac{\pi^2}{\gamma \sqrt{2}} \right)^{q-1} = \frac{\pi^2}{M}  \left( \frac{\pi^2}{\gamma \sqrt{2}} \right)^{q-1}.
\end{align*}
A similar bound can be determined for the second integral, which gives
$$
\int_{ A_j} \prod_{r=1}^{q} \frac{1}{\gamma^2 + \lambda_r^2}  \mathrm{d} \lambda \leq \frac{2}{M} \left( \frac{\pi}{\gamma} \right)^{q-1}.
$$
As a consequence, 
\begin{align}
	I_{M,n}^c \leq C_{q} \gamma^{-(q-1)} M^{-1} \quad \text{ for $ n \geq 2 \gamma$ and some constant $C_q>0$}.  \label{eq:I_{M,n}^c}
\end{align}
The upper bound for $ \E \big( W_{n,q}(1)	- W_{q,\gamma}(1) \big)^2$, namely 
$$
2 \big(K_{q,n}^{(\gamma)} - K_{q}^{(\gamma)} \big)^2  \left( \frac{\pi}{\gamma} \right)^q + 2 \big( K_{q}^{(\gamma)} \big)^2 \int_{ \R^q \setminus [- n \pi, n \pi]^q  } \left|  \psi(\lambda_1, \ldots, \lambda_q,1) \right|^2 \mathrm{d} \lambda_1 \ldots \mathrm{d} \lambda_q  \leq C_{\gamma,q} n^{-1},
$$
for some constant $C_{\gamma,q}>0$ is a consequence of Minkowski's inequality and  Remark \ref{remark:int}. 
To extract the $n$-dependence in the last bound as well as \eqref{eq:I_{M,n}} and \eqref{eq:I_{M,n}^c}, we obtain by Theorem~4.4 in \cite{NourdinPoly2013} and the isometry for multiple Wiener-It\^{o} integrals (cf.~\cite{PeccatiTaqqu2011}, equation (5.5.62))
	\begin{align*}
		&d_{TV} \left( \mathcal{L}\big(\hat{S}_{n,1}(H_q)\big) , \mathcal{L}(W_{q,\gamma}(1)\big)  \right) \\
		&\leq d_{TV} \left( \mathcal{L}\big(\hat{S}_{n,1}(H_q)\big) , \mathcal{L}(W_{n,q}(1) \big)  \right) +  d_{TV} \left( \mathcal{L}(W_{n,q}(1) \big) , \mathcal{L}(W_{q,\gamma}(1)\big)  \right) \\
		&\leq (K_q^{(\gamma)})^{\frac{1}{2q}} C_{\gamma,q} \left( n^{-2} + \frac{(1+M^4)^2}{n^4} + \frac{1+M^4}{n^3} + M^{-1} + n^{-1} \right)^{\frac{1}{4q}}
	\end{align*}
	for some constant $C_{\gamma,q}>0$. Now, recall that the bound for $I_{M,n}$ is valid under the conditions $n \geq M$ and $n \geq 1 + M$, while the bound for $I_{M,n}^c$ is valid under $n \geq 2 \gamma$. The choice $M=n^{4/9}$ gives finally
	\begin{align*}
		d_{TV} \Big( \hat{S}_{n,1}(H_q) - W_{q,\gamma}(1) \Big) &\leq (K_q^{(\gamma)})^{\frac{1}{2q}} C_{\gamma,q} \, n^{-\frac{1}{9q}} (1 + o_n(1)).
	\end{align*}

\end{proof}

\begin{Remark}
	\label{remark:beta=1}
	For $\psi_n$ in \eqref{psi_n} and $$ \big(\tilde{S}_{n,t}(H_q) \big)_{t \in [0,1]}= \bigg( K_{q,n}^{(\gamma)} \int_{[- n \pi, n \pi]^q}^{''} \psi_n(\lambda_1, \ldots , \lambda_q,t)  \, \mathrm{d} \tilde{B}_{0} \left( \lambda_1 \right) \ldots \mathrm{d} \tilde{B}_{0} \left( \lambda_q \right) \bigg)_{t \in [0,1]}$$ 
	the previous proof even reveals $\E \big(   \tilde{S}_{n,t}(H_q)   -      W_{q,\gamma}(t) \big)^2 \rightarrow 0$ as $n \to \infty$ for any $t \in [0,1]$.
	
\end{Remark}

\begin{Theorem}
	\label{theorem:beta=1}    
	Let $\beta=1$. 
	\begin{itemize}
		\item[(i)] If $\sigma_n^2 \gg  (1 - \alpha_n^2) $ consider $f=\sum_{q=m}^p c_q H_q$. Then
		$$\left( \hat{S}_{n,t}(f) \right)_{t \in [0,1]} \longrightarrow_{\mathcal{D}}   \biggl(\, \sum_{l=0}^{\lfloor (p-1)/2 \rfloor} e_{p-2l} W_{p-2l,\gamma}(t)\biggr)_{t\in[0,1]} \hspace{1mm} \text{in} \; D[0,1]  \text{ as }n\rightarrow\infty;$$
		\item[(ii)] If $\sigma_n^2 =  (1 - \alpha_n^2) $ consider $f \in L^2 \left(\mathcal{N} \left(0, 1 \right) \right)$. Then
		$$\left( \hat{S}_{n,t}(f) \right)_{t \in [0,1]} \longrightarrow_{\mathcal{D}}   \bigg( \, \sum_{q=m}^\infty d_q W_{q,\gamma}(t) \bigg)_{t \in [0,1]} \hspace{1mm} \text{in} \; D[0,1]  \text{ as }n\rightarrow\infty;$$
		\item[(iii)] If $\sigma_n^2 \ll  (1 - \alpha_n^2) $ consider $f=\sum_{q=m}^p c_q H_q$. Then
		$$ \left( \hat{S}_{n,t}(f) \right)_{t \in [0,1]} \longrightarrow_{\mathcal{D}}  
		\biggl( \, \sum_{q=1}^{{m_*}} g_{q} W_{q,\gamma}(t)\biggr)_{t\in[0,1]} \hspace{1mm} \text{in} \; D[0,1]  \text{ as }n\rightarrow\infty.
		$$
	\end{itemize}
\end{Theorem}

\begin{proof}
	 According to Remark \ref{remark:2.2}, it is sufficient to restrict attention to $Z_{i,n}$, $i=1,\ldots,n$, satisfying \eqref{Z_intro} with $\sigma_n^2=(1 - \alpha_n^2)$. Analogously to the case $\beta<1$, the proof is decomposed into three steps. For $f=H_q$, the convergence	$\hat{S}_{n,t}(H_q)  \longrightarrow_{\mathcal{D}} W_{q,\gamma}(t)$ as $n \to \infty$ for every $q \in \N$ and  $t \in [0,1]$ is a consequence of Theorem  \ref{prop: QLT} (ii) (replacing in there $\gamma$ by $\gamma t$ and $n$ by $\lfloor nt \rfloor$).
	Based on Remark \ref{remark:beta=1}, the convergence of the fidis $ \big(  \hat{S}_{n,t_1}(H_q), \ldots,  \hat{S}_{n,t_d}(H_q) \big)$   follows by the Cramér-Wold Theorem and Minkowski's inequality.
	For $\kappa_1, \ldots , \kappa_d \in \R$,
	\begin{align*}
		\bigg( \sum_{j=1}^d \kappa_j \hat{S}_{n,t_j}(H_m), \ldots, \sum_{j=1}^d \kappa_j \hat{S}_{n,t_j}(H_p) \bigg) 	\longrightarrow_{\mathcal{D}}  \bigg( \sum_{j=1}^d \kappa_j W_{m,\gamma}(t_j), \ldots, \sum_{j=1}^d \kappa_j W_{p,\gamma}(t_j) \bigg)
	\end{align*}
	follows by  Lemma 4.5 of \cite{BaiTaqqu2018}. 
	Rewrite $\hat{S}_{n,t} (f) =\sum_{q=m}^p w_{n,q} \hat{S}_{n,t}(H_q)$ with $w_{n,q}$ given in \eqref{eq:w_{n,q}}. Lemma~\ref{lemma: variance} reveals
	\begin{align*}
		w_{n,q} 
		&=   \sqrt{\frac{2 c_q^2  q!   \frac{ n^{2}   }{q^2 \gamma^2} \left( \exp(-q \gamma ) -1 + q \gamma  +o_n(1)   \right)}{  2 \sum_{r=m}^p  c_r^2 r!   \frac{ n^{2}   }{r^2 \gamma^2} \left( \exp(-r \gamma ) -1 + r \gamma  +o_n(1)   \right)    } } \\
		&\xrightarrow[]{n \to \infty}  \sqrt{\frac{ \frac{ c_q^2   q!  }{q^2}  \left( \exp(-q \gamma ) -1 + q \gamma    \right)}{   \sum_{r=m}^p \frac{ c_r^2 r!   }{r^2}  \left( \exp(-r \gamma ) -1 + r \gamma    \right)    } } =: d_q.
	\end{align*}
	From here the proof follows the one of Section \ref{sec: beta<1}.
	As $W_{q,\gamma}$ possesses a continuous version (cf. \cite{Sabzikar2015}), the statement then follows by Theorem 13.5 in \cite{Billingsley1999}, whence
	\begin{align}
		\left(   \hat{S}_{n,t} (f)  \right)_{t \in [0,1]} &=  \left(   \sum_{q=m}^p w_{n,q} \, \hat{S}_{n,t} (H_q)  \right)_{t \in [0,1]} \longrightarrow_\mathcal{D} \left(  \sum_{q=m}^p d_q \, W_{q,\gamma}(t) \right)_{t \in [0,1]}  \label{beta=1_combi}
	\end{align}
	in $D[0,1]$ as $n \to \infty$.
	For cases (i) and (iii) the proof is analogous. 
	In case (ii), the extension to  $f \in L^2 \left( \mathcal{N} (0, 1) \right)$ is conducted analogously to Step 3 in the proof of Theorem \ref{theorem:beta<1} and therefore omitted.
\end{proof}

\begin{Theorem}
	\label{theorem:beta=1_fn}    
	Let $\beta=1$ and $f_n=\sum_{q=m}^p c_q  H_q(\cdot;\mathrm{Var} (\cdot))$. 
	\begin{itemize}
		\item[(i)] If $\sigma_n^2 \gg  (1 - \alpha_n^2)$ then
		$$\left( \hat{S}_{n,t}(f_n) \right)_{t \in [0,1]} \longrightarrow_{\mathcal{D}}   \big( W_{p,\gamma}(t) \big)_{t\in[0,1]} \hspace{1mm} \text{in} \; D[0,1]  \text{ as }n\rightarrow\infty;$$
		\item[(ii)] If $\sigma_n^2 \ll  (1 - \alpha_n^2)$ then
		$$\left( \hat{S}_{n,t}(f_n) \right)_{t \in [0,1]} \longrightarrow_{\mathcal{D}}   \big( W_{m,\gamma}(t) \big)_{t\in[0,1]} \hspace{1mm} \text{in} \; D[0,1]  \text{ as }n\rightarrow\infty;$$
	\end{itemize}
\end{Theorem}
\begin{proof}
	In view of Lemma \ref{lemma:2.4}, the proof is analogous to the one of Theorem~\ref{theorem:beta=1} and hence omitted.
\end{proof}

\section{Approximations for {\boldmath$n\ll t_{\text{mix},n}$}}
\label{sec: beta>1}

Recall that $n \ll t_{\text{mix},n} $ if and only if $\beta > 1$.

\begin{proof}[Proof of Theorem \ref{prop:gamma} (iii)]
	The proof will make use of the following identity. 
\begin{Lemma}
	\label{lemma:5.2}
	For any $i=1,\dots, n$,
	\begin{align*}
		H_q \left( Z_{i,n} \right) =  \alpha_n^{qi}  H_q \left( Z_{0,n} \right) + R_q  \left( Z_{i,n} \right)
	\end{align*}
	with 
	\begin{align}
		R_q  \left( Z_{i,n} \right) = \sum_{r=0}^{q-1} \binom{q}{r} \alpha_n^{ri} \left( \sqrt{ 1 - \left( \alpha_n^{i} \right)^2 } \right)^{(q-r)}  \hspace{-2mm} H_r \left( Z_{0,n} \right) H_{q-r} \left(   Y_{i,n}  \right) \label{eq:5.1}
	\end{align}
	and $Y_{i,n}=  \sqrt{ \frac{ 1 - \alpha_n^2}{ 1 - \left( \alpha_n^{i} \right)^2}} \displaystyle\sum_{k=1}^{i} \varepsilon_k \, \alpha_n^{i-k}  $.
\end{Lemma}
\begin{proof}[Proof of Lemma \ref{lemma:5.2}]
	Using that
	$ X_{j,n}=\alpha_n^j X_{0,n} + \sqrt{1 - \alpha_n^2} \sum_{l=1}^j \varepsilon_l \, \alpha_n^{j-l}$, the proof follows from 	the identity
	$$ H_q(ax+by)=\sum_{k=0}^q \binom{q}{k} a^k \, b^{q-k} H_k(x) H_{q-k}(y) \quad \forall x,y \in \R \text{ and } a,b \text{ with } a^2+b^2=1,$$
	cf.~Lemma 2.2 in \cite{Hairer2021}. 
\end{proof}
By refining the result of Lemma \ref{lemma: variance}, one obtains by standard algebra that 
	$$
	\mathrm{Var}S_{n,1}(H_q) = q! \, n^2 \left( 1 - \frac{q \gamma}{3} n^{1-\beta} + o_n \big( n^{1-\beta} \big) \right).
	$$
	Denote $\sigma_{n,1}^2(H_q):=\mathrm{Var}S_{n,1}(H_q)$. For $q=1$ holds $H_1(x) =x$ and by Lemma \ref{lemma:5.2},
\begin{align*}
	\sum_{i=1}^{n }   H_1 \big( Z_{i,n} \big) = H_1 \big( Z_{0,n} \big)  \sum_{i=1}^{ n } \alpha_n^i +  \sqrt{1 - \alpha_n^2} \, \sum_{i=1}^{ n}  \sum_{l=1}^i    \varepsilon_l \, \alpha_n^{i-l}.
\end{align*}
The claim then follows by observing that for $\beta>1$ 
	$$ \E \Bigg( H_1 \big( Z_{0,n} \big) \bigg[ \frac{  \sum_{i=1}^{ n} \alpha_n^i }{ \sigma_{n,1}(H_1) } -1 \bigg] \Bigg)^2  \leq     C_\gamma  n^{2-2\beta} (1 + o_n(1))  $$ 
	and
\begin{align*}
	\frac{1}{\sigma_{n,1}^2(H_1)}\E \left(     \sqrt{1 - \alpha_n^2} \, \sum_{i=1}^{ n}  \sum_{l=1}^i    \varepsilon_l \, \alpha_n^{i-l} \right)^2 \leq  C \gamma n^{1-\beta} (1 + o_n(1)).
\end{align*}
 for constants $C_\gamma, C>0$. For $q \geq 2$ Lemma \ref{lemma:5.2} gives
\begin{align}
	\hat{S}_{n,1}(H_q) =   \frac{1}{\sigma_{n,1}(H_q)}  H_q \big( Z_{0,n} \big)  \sum_{i=1}^{ n}   \alpha_n^{qi}  + R_{n,q}(1), \label{decomp}
\end{align}
with $R_{n,q}(1)= \sum_{i=1}^{ n}  R_q  \big( Z_{i,n} \big)   / \sigma_{n,1}(H_q) $. Similarly to the case $q=1$, observe that
	$$ \E \Bigg( H_q \big( Z_{0,n} \big) \bigg[ \frac{  \sum_{i=1}^{ n} \alpha_n^{qi} }{ \sigma_{n,1}(H_q) } - \frac{1}{\sqrt{q!}} \bigg] \Bigg)^2  \leq     C_{\gamma,q}  n^{2-2\beta} (1 + o_n(1))  $$ 
for a constant $C_{\gamma,q} >0$.
We now derive an explicit expression for $R_{n,q}(1)$. 
First, consider that $\E R_{n,q}(1)= 0$ for any $q\geq2$, $n \in \N$  and 
\begin{align}
	\E \left( R_{n,q}(1)  \right)^2 = \frac{1}{\sigma_{n,1}^2(H_q)} \sum_{i=1}^{n}  \sum_{j=1}^{n}  \E \left[  R_q  \left( Z_{i,n} \right)  R_q  \left( Z_{j,n} \right) \right]  \label{eq:5.3}
\end{align}
by Lemma \ref{lemma: variance}. Since $H_l \indep H_m$ for $l \neq m$ and $Z_{0,n} \indep \big(Y_{i,n} \big)_{i \in \N}$, with $Y_{i,n}$ as specified in Lemma \ref{lemma:5.2}, $ \E \big[  R_q  \big( Z_{i,n} \big)  R_q  \big( Z_{j,n} \big) \big] $ is equal to
\begin{align*}
	\sum_{r=0}^{q-1} {\binom{q}{r}}^2   \alpha_n^{r(i+j)} \left(  1 -  \alpha_n^{2i}   \right)^{\frac{q-r}{2}}  \left(  1 -  \alpha_n^{2j}   \right)^{\frac{q-r}{2}} \E   \big( H_r \big( Z_{0,n} \big) \big)^2  \,  \E \left[  H_{q-r} \left(    Y_{i,n}  \right)  H_{q-r} \left(    Y_{j,n}  \right)  \right].
\end{align*}
Using  that $Z_{0,n}, Y_{i,n}, Y_{j,n} \sim \mathcal{N}(0,1)$, Mehler's formula reveals $    \E   \big( H_r \big( Z_{0,n} \big) \big)^2 = r! $ and
\begin{align*}
	\E \left[  H_{q-r} \left(    Y_{i,n}  \right)  H_{q-r} \left(    Y_{j,n}  \right)  \right]  &=     (q-r)! \left( \E \left[ Y_{i,n} Y_{j,n} \right] \right)^{q-r} \\  
	&= (q-r)! \Bigg( \frac{ 1 -  \alpha_n^2 }{\sqrt{ 1 -  \alpha_n^{2i}  } \sqrt{ 1 -  \alpha_n^{2j}  }}   \sum_{k=1}^{\min{\{i,j\}}} \alpha_n^{i+j-2k} \Bigg)^{q-r} .
\end{align*}
The sum is symmetric in $(i,j)$ and thus
\begin{align*}
	\sum_{k=1}^{\min{\{i,j\}}} \alpha_n^{i+j-2k} = \alpha_n^{|i-j|} \frac{1 - \alpha_n^{2\min{\{i,j\}}}}{1 - \alpha_n^2}.
\end{align*}
Plugging the latter expression into \eqref{eq:5.3}, we arrive at 
\begin{align*}
	\E \left[  R_q  \left( Z_{i,n} \right)  R_q  \left( Z_{j,n} \right) \right] = \sum_{r=0}^{q-1} {\binom{q}{r}}^2 r!  (q-r)! \alpha_n^{r(i+j)} \alpha_n^{(q-r)|i-j|} \left( 1 - \alpha_n^{2\min{\{i,j\}}} \right)^{q-r} 
\end{align*}
and $ \E  R_{n,q}^2(1) $ equals
\begin{align*}
	\frac{1}{\sigma_{n,1}^2(H_{q})} \sum_{i=1}^{n}  \sum_{j=1}^{ n}   \sum_{r=0}^{q-1} {\binom{q}{r}}^2 r!  (q-r)! \alpha_n^{r(i+j)} \alpha_n^{(q-r)|i-j|} \left( 1 - \alpha_n^{2\min{\{i,j\}}} \right)^{q-r} .
\end{align*}
Consider now the inner sum
\begin{align}
	&\sum_{r=0}^{q-1} {\binom{q}{r}}^2 r! \, (q-r)! \, \alpha_n^{r(i+j)} \alpha_n^{(q-r)|i-j|} \left( 1 - \alpha_n^{2\min{\{i,j\}}} \right)^{q-r} \notag \\
	&=  {\binom{q}{q-1}}^2 (q-1)! \,   \alpha_n^{(q-1)(i+j)} \alpha_n^{|i-j|} \left( 1 - \alpha_n^{2\min{\{i,j\}}} \right) \label{eq:5.4} \\
	&\hspace{5mm} + \sum_{r=0}^{q-2} {\binom{q}{r}}^2 r!  \,(q-r)! \alpha_n^{r(i+j)} \alpha_n^{(q-r)|i-j|} \left( 1 - \alpha_n^{2\min{\{i,j\}}} \right)^{q-r}. \label{eq:5.5}
\end{align}
Both terms are positive. First evaluate the contribution of \eqref{eq:5.5},	by observing that
\begin{align}
	&\frac{1}{\sigma_{n,1}^2(H_{q})} \sum_{i=1}^{n}  \sum_{j=1}^{n} \sum_{r=0}^{q-2}  \alpha_n^{r(i+j)} \alpha_n^{(q-r)|i-j|}  \left( 1 - \alpha_n^{2\min{\{i,j\}}} \right)^{q-r} \notag  \\
	&\hspace{5mm} \leq  \frac{q-1}{\sigma_{n,1}^2(H_{q})} \sum_{i=1}^{n}  \sum_{j=1}^{n}  \left( 1 - \alpha_n^{2\min{\{i,j\}}} \right) . \label{eq:5.8}
\end{align} 
To provide the rate of convergence, observe that
$$
1 - \alpha_n^{2\min{\{i,j\}}}=\frac{2 \gamma \min{\{i,j\}}}{n^\beta} (1 + o_n(1))
$$
and $\sum_{i=1}^{ n} \sum_{j=1}^{ n } \min{\{i,j\}}  \leq  n^{3}
$. Then, together with equation \eqref{eq:2_Lambda} and Lemma \ref{lemma: variance}, we get the following bound for the contribution of \eqref{eq:5.5} to $\E R_{n,q}^2(1)$, namely
\begin{align*}
\frac{q-1}{\sigma_{n,1}^2(H_{q})} \sum_{i=1}^{n}  \sum_{j=1}^{n}  \left( 1 - \alpha_n^{2\min{\{i,j\}}} \right)  \leq C(q) \, \gamma \, n^{ 1-\beta} (1+o_n(1))
\end{align*}
  for some constant $C(q)>0$. Due to symmetry in $(i,j)$, the contribution of the term \eqref{eq:5.4} to $\E R_{n,q+1}^2(t)$ is equal to 
\begin{align*}
	&\frac{1}{\sigma_{n,1}^2(H_{q})} \sum_{i=1}^{n}  \sum_{j=1}^{n}  \underbrace{{\binom{q}{q-1}}^2 (q-1)!  }_{=:D_q} \alpha_n^{(q-1)(i+j)} \alpha_n^{|i-j|} \left( 1 - \alpha_n^{2\min{\{i,j\}}} \right) \\
	&= \frac{D_q}{\sigma_{n,1}^2(H_{q})} \Bigg[ \sum_{i=1}^{n}  \alpha_n^{2(q-1)i} \left( 1 - \alpha_n^{2i} \right) +  2 \sum_{j=2}^{n}  \sum_{i=1}^{j-1}  \left(  \alpha_n^{i(q-2)}  \alpha_n^{jq}  - \alpha_n^{iq} \alpha_n^{jq}  \right) \Bigg].
\end{align*}
Now, elementary algebra reveals
\begin{align}
	&\sum_{i=1}^{n}  \alpha_n^{2(q-1)i} \left( 1 - \alpha_n^{2i} \right) = \mathcal{O}_n \big( n^{2-\beta} \big) \quad \text{ and } \notag \\
	&\sum_{j=2}^{n}  \sum_{i=1}^{j-1}  \left(  \alpha_n^{i(q-2)}  \alpha_n^{jq}  - \alpha_n^{iq} \alpha_n^{jq}  \right) = \mathcal{O}_n \big( n^{3-\beta} \big). \label{eq:5.6b}
\end{align}
Likewise,
$$
\frac{1}{\sigma_{n,1}^2(H_{q})} \sum_{i=1}^{ n}  \sum_{j=1}^{n}   \alpha_n^{(q-1)(i+j)} \alpha_n^{|i-j|} \left( 1 - \alpha_n^{2\min{\{i,j\}}} \right) \leq  C(q) \, \gamma \, n^{ 1-\beta} (1+o_n(1))
$$
for some constant $C(q)>0$.
So we directly obtain
\begin{align*}
	\E \left( \hat{S}_{n,1}(H_q) -    \frac{H_q \big( Z_{0,n}  \big)  }{\sqrt{q!}} \right)^2 \leq C(q) \, \gamma \, n^{ 1-\beta} (1+o_n(1)).
\end{align*}
From Theorem 4.4 in \cite{NourdinPoly2013} and the isometry for multiple Wiener-It\^{o} integrals (cf.~equation (5.5.62) in \cite{PeccatiTaqqu2011}), we get the bound 
	\begin{align*}
		d_{TV} \left( \mathcal{L}\big(\hat{S}_{n,1}(H_q)\big) ,   \mathcal{L}\Big(\frac{H_q \left( Z   \right) }{\sqrt{q!}}\Big) \right)    &\leq C_q \Bigg( \frac{1}{q!} \,  \E  \bigg( \hat{S}_{n,1}(H_q) -  \frac{H_q \big(Z_{0,n}\big)}{\sqrt{q!}} \bigg)^2 \Bigg)^{\frac{1}{4q}},
	\end{align*}
 which delivers the claim.
\end{proof}
On basis of Theorem \ref{prop: QLT} (iii) we now prove the functional limit theorem.
\begin{Theorem}
	\label{theorem:beta>1}    
	Let $\beta>1$. 
	\begin{itemize}
		\item[(i)] If $\sigma_n^2 \gg  (1 - \alpha_n^2) $ consider $f=\sum_{q=m}^p c_q H_q$. Then 
		$$ \left( \hat{S}_{n,t}(f) \right)_{t \in [0,1]}  \longrightarrow_{\mathcal{D}} \left( t \sum_{l=0}^{\lfloor (p-1)/2 \rfloor} f_{p-2l} \frac{H_{p-2l}(Z)}{\sqrt{(p-2l)!}} \right)_{t \in [0,1]}   \hspace{1mm} \text{in} \; D[0,1]  \text{ as }n\rightarrow\infty;$$
		\item[(ii)] If $\sigma_n^2 =  (1 - \alpha_n^2) $ consider $f \in L^2 \left(\mathcal{N} \left(0, 1 \right) \right)$. Then
		$$\left( \hat{S}_{n,t}(f) \right)_{t \in [0,1]} \longrightarrow_{\mathcal{D}} \left( t \sum_{q=m}^{\infty}
		h_q \frac{H_q(Z)}{\sqrt{q!}}  \right)_{t \in [0,1]}   \hspace{1mm} \text{in} \; D[0,1]  \text{ as }n\rightarrow\infty;$$
		\item[(iii)] If $\sigma_n^2 \ll  (1 - \alpha_n^2) $ consider $f=\sum_{q=m}^p c_q H_q$. Then
		$$ \left( \hat{S}_{n,t}(f) \right)_{t \in [0,1]}  \longrightarrow_{\mathcal{D}} \left( t \sum_{q=1}^{{m_*}} l_{q} \frac{H_{q}(Z)}{\sqrt{q!}} \right)_{t \in [0,1]}    \hspace{1mm} \text{in} \; D[0,1]  \text{ as }n\rightarrow\infty.$$
	\end{itemize}
\end{Theorem}
\begin{proof}
	Analogously to the cases $\beta<1$ and $\beta=1$, it is sufficient to restrict attention to the standardized variables $Z_{i,n}$, $i=1,\ldots,n$, according to Remark \ref{remark:2.2}. 
	 The convergence of the one-dimensional marginals for $f=H_q$ in $L^2$ as $n \to \infty$ for every $q \in \N$ and  $t \in [0,1]$ is given by Theorem \ref{prop: QLT} (iii)  (replacing in there $\gamma$ by $\gamma t$ and $n$ by $\lfloor nt \rfloor$).
	Now consider $0 \leq t_1 \leq \ldots \leq t_d \leq 1$ and $\kappa_1, \ldots, \kappa_d \in \R$ for $d \in \N$. Then 
	\begin{align*}
		\sum_{j=1}^d \kappa_j \hat{S}_{n,t_j}(H_q)  - \sum_{j=1}^d \kappa_j  t_j  \frac{H_q \big( Z_{0,n}  \big)  }{\sqrt{q!}}  \longrightarrow_{L^2} 0 \qquad \text{ as } n \to \infty
	\end{align*}
	by the triangle inequality, which by the Cramér-Wold Theorem implies the convergence of the fidis,  i.e.
	$$
	\big( \hat{S}_{n,t_1}(H_q), \ldots, \hat{S}_{n,t_d}(H_q) \big) \longrightarrow_{\mathcal{D}} \Bigg(t_1  \frac{H_q \big( Z  \big)  }{\sqrt{q!}}, \ldots, t_d  \frac{H_q \big( Z  \big)  }{\sqrt{q!}}\Bigg)  \qquad \text{ as } n \to \infty.
	$$
	Similarly to the Sections \ref{sec: beta<1} and \ref{sec: beta=1}, consider for $f=\sum_{q=m}^p c_q H_q$ the statistics
	$$ \hat{S}_{n,t}(f) =\sum_{q=m}^p w_{n,q} \, \hat{S}_{n,t}(H_q).$$ 
	Here, it holds as $n \to \infty$,
	\begin{align*}
		w_{n,q}
		\overset{\ref{lemma: variance}}{=}   \sqrt{\frac{ c_q^2  q!   n^2 (1 + o_n(1) )}{   \sum_{r=m}^p  c_r^2 r!   n^2 (1 + o_n(1) )} } 
		\xrightarrow[]{n \to \infty}   \sqrt{\frac{ c_q^2  q!   }{   \sum_{r=m}^p  c_r^2 r!     }}  =: h_q .
	\end{align*}
	Then for each $t \in [0,1]$ it holds by the orthogonality of Wiener chaos decomposition
	\begin{align*}
		\E &\Bigg( \hat{S}_{n,t}(f) - \sum_{q=m}^p h_q \, t \,  \frac{H_q \big( Z_{0,n} \big)}{\sqrt{q!}} \Bigg)^2\\ &\ \leq \sum_{q=m}^p \E \Bigg( w_{n,q} \, \hat{S}_{n,t}(H_q) -  h_q \, t \,   \frac{H_q \big( Z_{0,n} \big)}{\sqrt{q!}} \Bigg)^2 \\
		&\ \leq \sum_{q=m}^p \Bigg[ \big( w_{n,q} - h_q \big)^2 + h_q ^2 \;  \E \Bigg(  \hat{S}_{n,t}(H_q) -  t \,   \frac{H_q \big( Z_{0,n} \big)}{\sqrt{q!}} \Bigg)^2 \, \Bigg] \ \xrightarrow[]{n \to \infty} 0 .
	\end{align*}
	The extension to the fidis is proven similarly by the triangle inequality and the Cramér-Wold device, while tightness is guaranteed again by Lemma \ref{lemma: tightness}. Thus, as $n \to \infty$,
	\begin{align}
		\left(   \hat{S}_{n,t} (f)  \right)_{t \in [0,1]}  \longrightarrow_{\mathcal{D}} \left( t \sum_{q=m}^p h_q \, \frac{H_q \big( Z \big)}{\sqrt{q!}} \right)_{t \in [0,1]}   \label{beta>1_combi}
	\end{align}
	in $D[0,1]$. For cases (i) and (iii) the proof is analogous.
	The extension to $f \in L^2 \left( \mathcal{N} (0, 1) \right)$ is conducted analogously to Step 3 in the proof of Theorem \ref{theorem:beta<1} and therefore omitted.
\end{proof}

\begin{Theorem}
	\label{theorem:beta>1_fn}    
	Let $\beta>1$ and $f_n=\sum_{q=m}^p c_q  H_q(\cdot;\mathrm{Var} (\cdot))$. 
	\begin{itemize}
		\item[(i)] If $\sigma_n^2 \gg  (1 - \alpha_n^2)$ then
		$$\left( \hat{S}_{n,t}(f_n) \right)_{t\in[0,1]} \longrightarrow_{\mathcal{D}}  \bigg( t \, \frac{H_{p}(Z)}{\sqrt{p!}}  \bigg)_{t\in[0,1]} \hspace{1mm} \text{in} \; D[0,1]  \text{ as }n\rightarrow\infty;$$
		\item[(ii)] If $\sigma_n^2 \ll  (1 - \alpha_n^2)$ then
		$$\left( \hat{S}_{n,t}(f_n) \right)_{t\in[0,1]} \longrightarrow_{\mathcal{D}}  \bigg( t \, \frac{H_{m}(Z)}{\sqrt{m!}}  \bigg)_{t\in[0,1]} \hspace{1mm} \text{in} \; D[0,1]  \text{ as }n\rightarrow\infty.$$
	\end{itemize}
\end{Theorem}
\begin{proof}
	In view of Lemma \ref{lemma:2.4}, the proof is analogous to the one of Theorem~\ref{theorem:beta>1} and hence omitted.
\end{proof}

\section{Continuity of the limit process in {\boldmath$\gamma$}}
\label{sec: continuity}

In this section, we show that the limit processes for $\beta=1$ depend continuously on $\gamma$ with respect to the topology of weak convergence. For this section we remove the assumption of $\gamma \in (0,1)$, as for any $\beta>0$ and  $\gamma>0$, there exists $n_0(\gamma,\beta)$ such that $\gamma/n^{\beta}<1$ for all $n\geq n_0(\gamma,\beta)$.

\begin{proof}[Proof of Theorem \ref{prop:gamma}]
	Note first that $W_{q,\gamma}$ is a continuous centered process. Moreover, employing its time-domain representation \eqref{eq:1.2}, consecutive application of the stochastic Fubini theorem (cf.~Theorem~2.1 in \cite{PipirasTaqqu2010}), Fubini's theorem, and equation (5.5.62) in \cite{PeccatiTaqqu2011} reveals the covariance structure
\begin{align}
	\mathrm{Cov} \left( W_{q,\gamma}(s), W_{q,\gamma}(t) \right) = \frac{1}{2} \frac{2 \, q \gamma  \min(s,t) + e^{- q \gamma t} + e^{- q \gamma  s} -  e^{- q \gamma  (t-s)} - 1 }{e^{- q \gamma } - 1 + q \gamma }. \label{eq:Cov2}
\end{align}

\noindent
(i) 
Observe that for $q=1$, we have $ W_{1,\gamma}(t), B(t) \sim \mathcal{N}(0,t)$ for each $t \in [0,1]$ and $\gamma$, whence it follows $d_{TV}(\mathcal{L}\big(W_{q,\gamma}(1)\big) , \mathcal{N}(0,1)=0$. Let $q \geq2$.
 Adopting the notation from Section~\ref{sec: preparatory} and the time-domain representation from \eqref{eq:1.2}, $W_{q,\gamma}(t) = I_q \left( f_{t,\gamma} \right)$ with $f_{t,\gamma} = K_q^{(\gamma)} g_{t,\gamma}$ for  the symmetric kernel $g_{t,\gamma} \in L^2( \R^q)$, given as $g_{t,\gamma}(x_1, \ldots, x_q):= \int_0^t \prod_{j=1}^q h_{\gamma,s} (x_j) \, \mathrm{d}s$ with $h_{\gamma,s}(x):= e^{- \gamma(s-x)} \mathbf{1}_{ \{ x \leq s \} }$. It clearly holds $\|f_{t,\gamma} \otimes_r f_{t,\gamma}\|^2 = \big( K_q^{(\gamma)} \big)^4 \|g_{t,\gamma} \otimes_r g_{t,\gamma}\|^2$. Exploiting the bound valid for general multiple Wiener-It\^{o} integrals from Theorem~11.4.3. in \cite{PeccatiTaqqu2011}, we get
\begin{align*}
	d_{TV} \left( \mathcal{L}\big( W_{q,\gamma}(t)\big) , \mathcal{N}(0,t) \right) \leq 2 q \sqrt{ \sum_{r=1}^{q-1} (r-1)!^2 \binom{q-1}{r-1}^4 (2q-2r)!   \| f_{t,\gamma} \tilde{\otimes}_r f_{t,\gamma} \|^2_{L^2\left( \R^{2q-2r} \right)} } .
\end{align*}
We now define 
\begin{align}
	\rho_\gamma(s,u) := \int_{\R} h_{\gamma,s}(x) h_{\gamma,u}(x) \, \mathrm{d}x = \int_{-\infty}^{\min(s,u)} e^{-\gamma(s-x)}e^{-\gamma(u-x)} \, \mathrm{d}x = \frac{e^{-\gamma|s-u|}}{2\gamma} \label{eq:rho}
\end{align}
and compute $g_{t,\gamma} \otimes_r g_{t,\gamma}$: fix $r\in\{1,\ldots,q-1\}$ and express for  $(\bm{x},\bm{y}) \in \R^{q-r} \times \R^{q-r}$,
\begin{align*}
	(g_{t,\gamma}\otimes_r g_{t,\gamma})(\bm{x},\bm{y}) := \int_{\R^r} g_{t,\gamma}(x_1,\ldots,x_{q-r},z_1,\ldots,z_r) g_{t,\gamma}(y_1,\ldots,y_{q-r},z_1,\ldots,z_r) \, \mathrm{d}z_1 \cdots \mathrm{d}z_r.
\end{align*}
Inserting the expression of $h$ and applying Fubini, we arrive at
\begin{align*}
	&(g_{t,\gamma} \otimes_r g_{t,\gamma})(\bm{x},\bm{y}) \\
	&= \int_0^t \int_0^t \left[\int_{\R^r}\prod_{i=1}^r h_{\gamma,s}(z_i)h_{\gamma,u}(z_i) \mathrm{d} \bm{z} \right] \left(\prod_{j=1}^{q-r} h_{\gamma,s}(x_j)\right) \left(\prod_{j=1}^{q-r} h_{\gamma,u}(y_j)\right) \mathrm{d}s \mathrm{d}u \\
	&= \int_0^t \int_0^t \big( \rho_\gamma(s,u) \big)^r \left( \prod_{j=1}^{q-r} h_{\gamma,s}(x_j) \right) \left( \prod_{j=1}^{q-r} h_{\gamma,u}(y_j) \right)\mathrm{d}s \mathrm{d}u.
\end{align*}
Now we calculate its $L^2$-norm. Integrating the square of last expression over
$(\bm{x},\bm{y})$, and using again Fubini's theorem leads to 
\begin{align*}
	&\| g_{t,\gamma} \otimes_r g_{t,\gamma} \|_{L^2(\R^{2q-2r})}^2 \\
	&= \int_{[0,t]^4} \big( \rho_\gamma(s,u) \big)^r \big( \rho_\gamma(s',u') \big)^r \big( \rho_\gamma(s,s') \big)^{q-r} \big( \rho_\gamma(u,u') \big)^{q-r} \mathrm{d} s    \mathrm{d} u \mathrm{d} s' \mathrm{d} u'\\
	&= (2\gamma)^{-2q} \int_{[0,t]^4} e^{  -\gamma \left(r|s-u|+r|s'-u'|+(q-r)|s-s'|+(q-r)|u-u'| \right)  } \mathrm{d} s    \mathrm{d} u \mathrm{d} s' \mathrm{d} u'.
\end{align*}
We now determine a bound for $\|g_{t,\gamma} \otimes_r g_{t,\gamma}\|^2$. Set
\begin{align*}
	I_{q,r}(\gamma;t) &:= \int_{[0,t]^4} e^{  -\gamma \left(r|s-u|+r|s'-u'|+(q-r)|s-s'|+(q-r)|u-u'| \right)  } \mathrm{d} s    \mathrm{d} u \mathrm{d} s' \mathrm{d} u'
\end{align*}
and consider $ J(s,s';\gamma) := \int_0^t \int_0^t  e^{  -\gamma \left(r|s-u|+(q-r)|u-u'|+r|s'-u'| \right)  }    \mathrm{d} u  \mathrm{d} u'$, such that $I_{q,r}(\gamma;t)=\int_0^t \int_0^t e^{-\gamma (q-r)|s-s'|} J(s,s';\gamma)\mathrm{d} s     \mathrm{d} s'$. For each fixed $s,s'$, bound the inner $u'$-integral with
\begin{align*}
	\int_0^t e^{-\gamma r|s'-u'|}e^{-\gamma (q-r)|u-u'|} \mathrm{d} u' &\leq \int_{\R} e^{-\gamma r|s'-u'|} e^{-\gamma (q-r)|u-u'|} \mathrm{d} u' \\
	&\leq \left( \int_{\R} e^{-2\gamma r|y|} \mathrm{d} y \right)^{1/2} \left(\int_{\R} e^{-2\gamma (q-r)|y|} \mathrm{d} y \right)^{1/2} \\
	&= \frac{1}{\gamma \sqrt{r(q-r)}}. 
\end{align*}
Thus,
\begin{align*}
	J(s,s';\gamma) \leq \int_0^t e^{-\gamma r|s-u|} \frac{1}{\gamma\sqrt{r(q-r)}} \mathrm{d} u \leq \frac{1}{\gamma\sqrt{r(q-r)}} \int_{\R} e^{-\gamma r|y|} \mathrm{d} y = \frac{2}{\gamma^2 r^{3/2}\sqrt{q-r}}.
\end{align*}
So we can provide the following estimate for all $\gamma \geq 1$:
\begin{align*}
	I_{q,r}(\gamma;t) &\leq \frac{2}{\gamma^2 r^{3/2}\sqrt{q-r}} \int_0^t\int_0^t e^{-\gamma (q-r)|s-s'|} \mathrm{d} s \mathrm{d} s' \\
	&\leq \frac{2}{\gamma^2 r^{3/2}\sqrt{q-r}} \int_{\R} e^{-\gamma (q-r)|y|} \mathrm{d} y = \frac{C_{q,r,t}}{\gamma^{3}}
\end{align*}
with $C_{q,r,t}=4 (r(q-r))^{-3/2}$.
Next, we can bound $\|f_{t,\gamma} \otimes_r f_{t,\gamma}\|^2 \leq \big( K_q^{(\gamma)} \big)^4 \frac{C_{q,r,t}}{2^{2q}} \gamma^{-(2q+3)}$. As shown in Section \ref{sec: beta=1},
\begin{align*}
	K_q^{(\gamma)} :=  \sqrt{\frac{q^2 \gamma^{q+2}}{q! \, 2^{1-q} \left(\exp(-\gamma q) - 1 + \gamma q \right)}} 
\end{align*}
with $ \left(\exp(-\gamma q) - 1 + \gamma q \right) \asymp q \gamma$ for $\gamma \to\infty$, hence
\begin{align*}
	\big( K_q^{(\gamma)} \big)^4 \asymp \left( \frac{2^{q-1}q}{q!} \right)^2 \gamma^{2q+2}, \qquad \gamma \to \infty.
\end{align*}
Combining the previous bounds, we can write by Lemma 14.22 in \cite{Kallenberg2021} for some constant $ C'_{q,r,t}>0$, 
\begin{align}
	\|f_{t,\gamma} \tilde{\otimes}_r f_{t,\gamma}\|^2  \leq \|f_{t,\gamma} \otimes_r f_{t,\gamma}\|^2 \leq C'_{q,r,t} \gamma^{2q+2} \gamma^{-(2q+3)} = C'_{q,r,t} \gamma^{-1}, \qquad \gamma \geq 1, \label{eq:6.1}
\end{align}
from which the statement follows.

\medskip
\noindent
(ii) 
The argument for $q=1$ is identical to the one in (i). Let $q\geq2$.
We show that 
\begin{align}
	\E \left[ \sup_{t \in[0,1]} \left| \sum_{q=1}^p c_q W_{q,\gamma}(t) - t\sum_{q=1}^p c_q \frac{H_q(Z)}{\sqrt{q!}} \right|^2 \, \right] \xrightarrow[\gamma \searrow 0] {} 0, \label{eq:sup}
\end{align}
providing a quantitative bound. Applying the stochastic Fubini theorem (cf.~Theorem~2.1 in \cite{PipirasTaqqu2010}) and considering $h_{\gamma,s}(x)= e^{- \gamma(s-x)} \mathbf{1}_{ \{ x \leq s \} }$,  we write
\begin{align*}
	W_{q,\gamma}(t)=K_q^{(\gamma)} \int_0^t T_{q,\gamma}(s) \mathrm{d}s, \qquad t \in [0,1],
\end{align*}
with 
\begin{align*}
	T_{q,\gamma}(s) := \int_{\R^q}^{'}  \prod_{j=1}^q h_{\gamma,s} (x_j) dB(x_1) \cdots dB(x_q) = I_q \big(h_{\gamma,s}^{\otimes q}\big), \qquad s \in [0,1].
\end{align*}
 Observe that	$W_{q,\gamma}(t) - Y_{q,\gamma}(t) = K_q^{(\gamma)} \int_0^t \big( T_{q,\gamma}(s) - T_{q,\gamma}(0) \big) \mathrm{d}s$, with $Y_{q,\gamma}(t) := K_q^{(\gamma)}  t T_{q,\gamma}(0)$. By Jensen's inequality and Fubini's theorem, we get
\begin{align}
	\E \left[ \sup_{t\in[0,1]} |W_{q,\gamma}(t)-Y_{q,\gamma}(t)|^2 \right] \leq \big( K_q^{(\gamma)} \big)^2 \int_0^1 \E \left[ \big( T_{q,\gamma}(s) -T_{q,\gamma}(0) \big)^2 \right] \mathrm{d}s . \label{eq:continuity_help_1}
\end{align}
By equation (5.5.62) in \cite{PeccatiTaqqu2011},  $\E [T_{q,\gamma}(s) T_{q,\gamma}(u)] = q! { \rho_\gamma(s,u)}^q$ for $s,u \geq 0$, with $\rho_\gamma(s,u)$ given in~\eqref{eq:rho}. Thus, by using $1-e^{-x} \leq x$, we  obtain
\begin{align*}
	\E \left[ \big( T_{q,\gamma}(s) - T_{q,\gamma}(0) \big)^2 \right] = \E [T_{q,\gamma}(s)^2] + \E [T_{q,\gamma}(0)^2] - 2 \E[T_{q,\gamma}(s) T_{q,\gamma}(0)] \leq 2 q!(2\gamma)^{-q} \gamma q s
\end{align*}
and by equation \eqref{eq:continuity_help_1}
\begin{align*}
	\E \left[\sup_{t\in[0,1]} |W_{q,\gamma}(t) - Y_{q,\gamma}(t)|^2 \right] \leq \big(K_q^{(\gamma)}\big)^2 2 q! (2\gamma)^{-q} \gamma q \int_0^1 s \mathrm{d} s = \big(K_q^{(\gamma)}\big)^2  q q! (2\gamma)^{-q} \gamma.
\end{align*}
We now observe that $\big(K_q^{(\gamma)} \big)^2  (2\gamma)^{-q}  q! = 1 + \frac{q}{3} \gamma + \mathcal{O}(\gamma^3)$ for $\gamma \searrow 0$. This yields 
\begin{align}
	\E \left[\sup_{t\in[0,1]} |W_{q,\gamma}(t) - Y_{q,\gamma}(t)|^2 \right] \leq 2q \gamma (1 + o_\gamma(1)), \label{eq:6.3}
\end{align}
where the Landau symbol $o_\gamma$ here is intended for $\gamma \searrow 0$. Now consider the normalized kernel $\bar h_{\gamma,0}(x) = \sqrt{2 \gamma}  h_{\gamma,0}(x)$ at time $t=0$, that is $\| \bar{h}_{\gamma,0}\|_{L^2(\R)}=1$, and define
\begin{align*}
	Z := I_1(\bar{h}_{\gamma,0}) = \int_{\R}   \bar{h}_{\gamma,0} (x) dB(x) =  \sqrt{2 \gamma}  \int_{-\infty}^0  e^{\gamma x}  dB(x) \sim \mathcal{N} (0,1).
\end{align*}
By It\^{o}'s polynomial representation (cf. Theorem 14.25 in \cite{Kallenberg2021}), it holds $ I_q(\bar{h}_{\gamma,0}^{\otimes q}) = H_q \left( I_1(\bar{h}_{\gamma,0}) \right)$, and by scaling we obtain the expression $Y_{q,\gamma}(t) = K_q^{(\gamma)} t I_q \big( h_{\gamma,0}^{\otimes q} \big) = t K_q^{(\gamma)}  (2 \gamma)^{-q/2} I_q \big(\bar{h}_{\gamma,0}^{\otimes q}\big) = t K_q^{(\gamma)}  (2 \gamma)^{-q/2} H_q(Z)$. Thus, we write
\begin{align*}
	\sup_{t\in[0,1]} \left| Y_{q,\gamma}(t) - \frac{t}{\sqrt{q!}} H_q(Z)\right|^2 = \left| K_q^{(\gamma)}  (2 \gamma)^{-q/2} - \frac{1}{\sqrt{q!}} \right|^2 \cdot |H_q(Z)|^2.
\end{align*}
By using  $\E [H_q(Z)^2] = q!$ and $K_q^{(\gamma)}  (2 \gamma)^{-q/2} - 1/ \sqrt{q!} = \frac{q}{6 \sqrt{q!}} \gamma + \mathcal{O}(\gamma^3)$, this reveals 
\begin{align}
	\E \left[\sup_{t\in[0,1]} \left| Y_{q,\gamma}(t) - \frac{t}{\sqrt{q!}} H_q(Z) \right|^2 \right]  = q! \left( K_q^{(\gamma)}  (2 \gamma)^{-q/2} - \frac{1}{\sqrt{q!}} \right)^2 = \frac{q^2}{36} \gamma^2 (1 + o_\gamma(1)). \label{eq:6.4}
\end{align}
	 Now, Theorem 4.4 in \cite{NourdinPoly2013} reveals
\begin{align*}
	d_{TV} \left( \mathcal{L}\big(W_{q,\gamma}(1)\big) , \mathcal{L} \left(  \frac{H_q(Z)}{\sqrt{q!}} \right) \right) \leq \tilde{C}_2(q) \Bigg( \E \bigg(  W_{q,\gamma}(1) - \frac{H_q(Z)}{\sqrt{q!}} \bigg)^2 \Bigg)^{\frac{1}{4q}} .
\end{align*}
Then the assertion follows by \eqref{eq:6.3} and \eqref{eq:6.4} as  $\gamma \to 0$.

\end{proof}

\begin{Theorem}
	\label{theorem:continuity}
	$W_{q,\gamma}$ is continuous in $\gamma$ with respect to the topology of weak convergence and for any nonnegative $c_m, \ldots , c_p$ with $\sum_{q=m}^p c_q^2=1$, 
	\begin{align*}
		\biggl(\sum_{q=m}^{p} c_q W_{q,\gamma}(t)\biggr)_{t\in[0,1]}    &\longrightarrow_{\mathcal{D}} \left(  B(t)  \right)_{t \in [0,1]}   \hspace{1mm} \text{in } \; C[0,1]  \text{ as } \gamma \rightarrow\infty, \\
		\biggl(\sum_{q=m}^{p} c_q W_{q,\gamma}(t)\biggr)_{t\in[0,1]}   &\longrightarrow_{\mathcal{D}}   \left( t   \sum_{q=1}^{p} c_q  \,   \frac{H_q \left( Z \right)}{\sqrt{q!}} \right)_{t \in [0,1]} \hspace{1mm} \text{in } \; C[0,1]  \text{ as } \gamma \rightarrow 0.
	\end{align*}
\end{Theorem}

\begin{proof}
(i) From Theorem \ref{prop:gamma} (i), $ W_{q,\gamma}(t) \longrightarrow_{\mathcal{D}} \mathcal{N} \left( 0, t \right)$ as $\gamma \to \infty$ for any $t\in[0,1]$.  Now define the linear combination process $$ S_{\gamma,t}^{(p)} := \sum_{q=m}^p c_q W_{q,\gamma}(t), \quad t \in [0,1].$$ Orthogonality of Wiener chaoses and an application of de L'H\^{o}pital's rule yield
\begin{align*}
	\mathrm{Cov} \left(S_{\gamma,t}^{(p)} ,S_{\gamma,s}^{(p)} \right) = \sum_{q=m}^p c_q^2\, \mathrm{Cov} \left(W_{q,\gamma}(t),W_{q,\gamma}(s)\right) \xrightarrow[]{\gamma \to \infty} \sum_{q=m}^p c_q^2 \min(s,t) =\min(s,t).
\end{align*}
 The convergence of the finite-dimensional marginals   $$\big( W_{q,\gamma}(t_1), \ldots ,W_{q,\gamma}(t_d) \big), \quad 0\leq t_1<\dots <t_d\leq 1,$$ is now a consequence of  Proposition 1 of \cite{PeccatiTudor2005}. Considering constants $\kappa_1, \ldots , \kappa_d \in \R$, Theorem 1 of \cite{PeccatiTudor2005} provides the convergence of 
  $$\bigg( \sum_{j=1}^d \kappa_j W_{m,\gamma}(t_j), \ldots , \sum_{j=1}^d \kappa_j W_{p,\gamma}(t_j) \bigg).$$ 
 Tightness can be shown analogously as in Lemma \ref{lemma: tightness}. Using the assumpion $\sum_{q=m}^p c_q^2=1$, analogously to Section \ref{sec: beta<1}, continuous-mapping theorem and Cramér-Wold device deliver the assertion.

\medskip
\noindent
(ii)
The statement is an immediate consequence of \eqref{eq:sup}.
\end{proof}

\bibliographystyle{apalike}
{\newpage\small \bibliography{references}  }

\newpage

\appendix
\appendixpage

\begin{appendix}
	\section{Discrete-time stationary Gauss-Markov process}
	The following lemma is well-known, yet we did not find an appropriate discrete-time reference in the literature and therefore provide a proof for the reader's convenience. It states that after centering, the class of stationary, Gaussian discrete-time Markov processes   only contains AR(1) processes. 
	
	\begin{Lemma}
		\label{lemma:AR1}
		Let $(X_k)_{k \in \N_0}$ be a stationary Gaussian Markov process. Then after centering with $\E X_0=\mu$, the process is either constant in time or satisfies the recursion
		$$
		X_{k} - \mu = \alpha (  X_{k-1} - \mu) + \varepsilon_k,\quad k\geq 1,
		$$
		with $\varepsilon_k \overset{\text{i.i.d.}}{\sim} \mathcal{N}\big(0, \mathrm{Var}(X_0) (1 - \alpha^2)\big)$ and $|\alpha|<1$.
	\end{Lemma}
	
	\begin{proof}
		Let $\mathrm{Var}(X_0)>0$ and $\alpha= \mathrm{Cov}(X_0,X_1)/\mathrm{Var}(X_0)$. Then $|\alpha|\leq 1$ by the Cauchy-Schwarz inequality and $\mathrm{Var}(X_1)=\mathrm{Var}(X_0)$. If $\alpha=1$, the process is constant in time whereas $\alpha=-1$ contradicts stationarity. Thus, it remains to study $|\alpha|< 1$.
		Proposition~14.7 in \cite{Kallenberg2021} reveals 
		$$
		\mathrm{Cov}(X_0,X_{k+1}) \mathrm{Var}(X_1) =  \mathrm{Cov}(X_0,X_1)  \mathrm{Cov}(X_1,X_{k+1}) .
		$$
		By stationarity, $\mathrm{Var}(X_1)=\mathrm{Var}(X_0)$ and $\mathrm{Cov}(X_1,X_{k+1})=\mathrm{Cov}(X_0,X_{k})$, hence
		\begin{align}
			\mathrm{Cov}(X_0,X_{k+1})=\frac{ \mathrm{Cov}(X_0,X_1)}{\mathrm{Var}(X_0)} \mathrm{Cov}(X_0,X_{k}) = \alpha \, \mathrm{Cov}(X_0,X_{k}). \label{eq:C.1}
		\end{align}
		Next, define $\varepsilon_k:=X_k-\mu - \alpha \, (X_{k-1} - \mu)$. We show that $\varepsilon_k \indep \{X_j: j \leq k-1 \}$. To this aim, note that  for $j \leq k-1$,
		$$
		\mathrm{Cov} (\varepsilon_k,X_j)=\mathrm{Cov}(X_k,X_j) - \alpha \, \mathrm{Cov}(X_{k-1},X_j)
		$$
		by definition of $\varepsilon_k$. Now, if $j = k-1$, stationarity reveals
		\begin{align*}
			\mathrm{Cov} (\varepsilon_k,X_{k-1}) &= \mathrm{Cov} (X_k,X_{k-1}) - \alpha \, \mathrm{Var}(X_{k-1}) \\
			&= \mathrm{Cov} (X_0,X_1) - \frac{\mathrm{Cov} (X_k,X_{k-1})}{\mathrm{Var}(X_{0})} \mathrm{Var}(X_{0}) = 0.
		\end{align*}
		If $j < k-1$, we obtain with  equation \eqref{eq:C.1},
		\begin{align*}
			\mathrm{Cov} (\varepsilon_k,X_j) &= \mathrm{Cov}(X_k,X_j) - \alpha \, \mathrm{Cov}(X_{k-1},X_j) \\
			&= \alpha \, \mathrm{Cov}(X_{k-1},X_j) - \alpha \, \mathrm{Cov}(X_{k-1},X_j) = 0.
		\end{align*}
		Since $\varepsilon_k$ and$ \{X_j: j \leq k-1 \}$ are jointly normally distributed, $\varepsilon_k$ is even independent of $\{X_j: j \leq k-1 \}$. $\varepsilon_k$ is clearly centered with variance
		\begin{align*}
			\mathrm{Var}(\varepsilon_k) &= \mathrm{Var} (X_0) + \alpha^2 \mathrm{Var} (X_0) - 2 \alpha \, \mathrm{Cov (X_0,X_1)} \\
			&= \mathrm{Var} (X_0) + \alpha^2 \mathrm{Var} (X_0) - 2 \alpha^2 \mathrm{Var} (X_0) = \mathrm{Var} (X_0) (1 - \alpha^2).
		\end{align*}
		As concerns  $l<k$, we write $\varepsilon_l=X_l - \mu - \alpha \, (X_{l-1} - \mu)$. Then $\varepsilon_l$ is represented as a function of $X_l $ and $X_{l-1}$, both in the past of time-step $k$. Due to previous observations, $\varepsilon_k \indep \varepsilon_l$ and more generally $\varepsilon_k \indep (\varepsilon_1, \ldots, \varepsilon_{k-1})$. By induction, the sequence $(\varepsilon_k)_{k \in \N}$ is i.i.d. 
		Finally, isolating $X_k$ in the defining equation for $\varepsilon_k$, one recovers the defining equation of an AR(1) process, which proves the assertion.
	\end{proof}

	\newpage
	\section{Mixing time}
	\begin{Lemma}
		\label{lemma:t_mix}
		Let $X_{1,n},\dots, X_{n,n}$ be observations satisfying recursion \eqref{Z_intro}. Then there exist constant $c,C >0$ such that
		$$
		c   \min \bigg\{ 1,  \frac{ |x| \alpha_n^j\sqrt{1-\alpha_n^2}}{\sigma_n}   +   \alpha_n^{2j} \bigg\} \leq \Big\| \mathcal{L}_x\big(X_{j,n}\big) - \pi_n\Big\|_{\mathrm{TV}} \leq C   \min \bigg\{ 1,  \frac{ |x| \alpha_n^j\sqrt{1-\alpha_n^2}}{\sigma_n}   +   \alpha_n^{2j} \bigg\}.
		$$
		In particular, $t_{\text{mix},n}\asymp n^{\beta}/\gamma$ for $\alpha_n=(1-\gamma/n^{\beta})$ with the local mixing time $t_{\text{mix},n}$ in \eqref{eq:t_mix}.
	\end{Lemma}
	
	\begin{proof}
		Starting the recursion \eqref{Z_intro} from a deterministic initial condition $X_0^{(n,x)} = x$, the $j$-th iteration equals
		\begin{align*}
			X_j^{(n,x)} = \alpha_n^j x + \sum_{k=1}^j \alpha_n^{j-k}  \sigma_n \varepsilon_k,
		\end{align*}
		whereas $\mathcal{L}\big(X_{j,n}\big)=\pi_n$ 	for the unique stationary solution. With the notation
		\begin{align*}
			&\mu_{n,j} := \mathcal{L}\big(X_j^{(n,x)}\big) =  \mathcal{N} \left(m_{n,j}, s_{n,j}^2\right), \qquad \text{ with } \quad m_{n,j} = \alpha_n^j x, \quad s_{n,j}^2 = \frac{\sigma_n^2}{1-\alpha_n^2} (1-\alpha_n^{2j}); \\
			&\pi_n = \mathcal{L}\big(X_{j,n}\big) =  \mathcal{N} \left(0, s_{n,\infty}^2 \right), \qquad \text{ with } \quad s_{n,\infty}^2 = \frac{\sigma_n^2}{1-\alpha_n^2},
		\end{align*}
		we determine an upper and lower bound for the  total variation distance of $\mu_{n,j} $ and $\pi_n$. For the upper bound, consider $ \nu_{n,j} := N \big(m_{n,j}, s_{n,\infty}^2\big)$, that is, the Gaussian measure with stationary variance but shifted mean $m_{n,j}$. By the triangle inequality,
		\begin{align*}
			\| \mu_{n,j} - \pi_n \|_{\mathrm{TV}}  \leq  \| \mu_{n,j} - \nu_{n,j} \|_{\mathrm{TV}} + \| \nu_{n,j} - \pi_n \|_{\mathrm{TV}}.
		\end{align*}
		Next, for any $\sigma^2>0$ and $m_1,m_2\in\R$, 
		\begin{align}
			\|N(m_1,\sigma^2) - N(m_2,\sigma^2)\|_{\mathrm{TV}} = 2\,\Phi\!\left(\frac{|m_1 - m_2|}{2\sigma}\right) - 1 \leq \frac{|m_1-m_2|}{\sqrt{2\pi}\,\sigma}, \label{eq:delta_mu}
		\end{align}
		where $\Phi$ denotes the distribution function of the standard normal distribution.
		Applying this inequality with $m_1 = m_{n,j} = \alpha_n^j x$, $m_2 = 0$, and $\sigma = s_{n,\infty}$,
		we obtain
		\begin{align*}
			\|\nu_{n,j} - \pi_n \|_{\mathrm{TV}} \leq   \frac{| \alpha_n^j x|}{\sqrt{2\pi} s_{n,\infty}} .
		\end{align*}
		For the distance of normal laws with different variances, recall that by the first Pinsker inequality (cf.~Lemma 2.5 (i) in \cite{Tsybakov2009}), $\|\mu_{n,j} - \nu_{n,j}\|_{\mathrm{TV}} \leq \sqrt{KL(\mu_{n,j},\nu_{n,j})/2}$, where $KL(\cdot,\cdot)$ is the Kullback-Leibler divergence. Elementary algebra reveals 
		$$
		KL(\mu_{n,j},\nu_{n,j}) = \log \bigg( \frac{s_{n,\infty}}{s_{n,j}} \bigg) - \frac{1}{2} + \frac{s_{n,j}^2}{2 s_{n,\infty}^2} = \frac{1}{2} \Big( - \log \big(1 - \alpha_n^{2j}\big) -\alpha_n^{2j}\Big).
		$$
		Expanding the logarithm provides $- \log \big(1 - \alpha_n^{2j}\big) = \alpha_n^{2j} + \alpha_n^{4j}/2 + \mathcal{O}_n (\alpha_n^{6j})$, whence $KL(\mu_{n,j},\nu_{n,j}) = \mathcal{O}_n (\alpha_n^{4j})$ and thus $ \|\mu_{n,j} - \nu_{n,j}\|_{\mathrm{TV}}  \leq  C \alpha_n^{2j}$. Combining the two bounds gives the claimed upper bound
		\begin{align}
			\|\mu_{n,j} - \pi_n\|_{\mathrm{TV}} \leq C_1 |x| \frac{ \alpha_n^j }{ s_{n,\infty}}  +  C_2 \alpha_n^{2j} \leq C \left(   |x| \frac{ \alpha_n^j }{ s_{n,\infty}}  +  \alpha_n^{2j} \right) . \label{eq:7.1}
		\end{align}
		As concerns the lower bound, set 
		\begin{align*}
			A= \begin{cases} [0, \infty), \quad x \geq 0, \\
				(-\infty,0), \quad x<0.
			\end{cases}
		\end{align*}
		Then $\|\mu_{n,j} - \pi_n\|_{\mathrm{TV}} \geq | \mu_{n,j}(A) - \pi_n(A)|$ by definition of total variation, where $\pi_n(A)=1/2$ for each value of $x$. For $\mu_{n,j}(A)$, first suppose that $x\geq 0$. Thus $m_{n,j}\geq 0$, $A=[0,\infty)$ and 
		$$
		\mu_{n,j}(A) = \Pro \big(X_j^{(n,x)}\geq 0\big) = \Pro \Bigg( \frac{X_j^{(n,x)} - m_{n,j}}{s_{n,j}} \geq \frac{- m_{n,j}}{s_{n,j}} \Bigg) = 1 - \Phi \bigg( \frac{- m_{n,j}}{s_{n,j}} \bigg)  = \Phi \bigg( \frac{ m_{n,j}}{s_{n,j}} \bigg) .
		$$
		If $x <0$, then $m_{n,j} < 0$, $A=[-\infty,0)$ and analogously
		$$
		\mu_{n,j}(A) = \Pro \big(X_j^{(n,x)} < 0\big) = \Pro \Bigg( \frac{X_j^{(n,x)} - m_{n,j}}{s_{n,j}} < \frac{- m_{n,j}}{s_{n,j}} \Bigg) =  \Phi \bigg( \frac{ - m_{n,j}}{s_{n,j}} \bigg) =  \Phi \bigg( \frac{ |m_{n,j}|}{s_{n,j}} \bigg) .
		$$
		Therefore, $\mu_{n,j}(A)= \Phi ( |m_{n,j}|/ s_{n,j} ) \geq 1/2$ for any $x \in \R$ and hence, 
		$$
		\|\mu_{n,j} - \pi_n\|_{\mathrm{TV}} \geq \Phi \bigg( \frac{ |m_{n,j}|}{s_{n,j}} \bigg) - \frac{1}{2} =  \Phi \bigg( \frac{ |x| \alpha_n^j}{s_{n,j}} \bigg) - \frac{1}{2} \geq \Phi \bigg( \frac{ |x| \alpha_n^j}{s_{n,\infty}} \bigg) - \frac{1}{2}.
		$$
		Now let $w:=(|x|\alpha_n^j)/s_{n,\infty}$ and $\phi(t)=\frac{1}{\sqrt{2 \pi}} e^{-t^2/2} $. If $0 \leq w \leq 1$, then
		$$
		\Phi (  w ) - \frac{1}{2} = \int_0^w \phi(t) \mathrm{d} t \geq \int_0^w  \phi(1)  \mathrm{d} t = \phi(1) w,
		$$
		while
		$$
		\Phi (  w ) - \frac{1}{2} \geq \Phi (  1 ) - \frac{1}{2} = \int_0^1 \phi(t) \mathrm{d} t \geq \int_0^1 \phi(1) \mathrm{d} t = \phi(1)
		$$ 
		if $w>1$. Hence, $\Phi (  w ) - \frac{1}{2} \geq \phi(1) \min \{1,w \}$ for any $w\geq 0$ and thus,
		\begin{align}
			\|\mu_{n,j} - \pi_n\|_{\mathrm{TV}} \geq \phi(1)  \min \bigg\{ 1,  \frac{ |x| \alpha_n^j}{s_{n,\infty}}  \bigg\}. \label{eq:7.2}
		\end{align}
		Now consider the set $B=[-s_{n,\infty},s_{n,\infty}]$	and the random variable $Y \sim \lambda_{n,j}:=\mathcal{N}(0,s_{n,j}^2)$. Then
		$$
		\lambda_{n,j}(B) = \Pro (-s_{n,\infty} \leq Y \leq s_{n,\infty}) = \Pro \bigg( - \frac{s_{n,\infty}}{s_{n,j}} \leq \frac{Y}{s_{n,j}} \leq \frac{s_{n,\infty}}{s_{n,j}} \bigg) = 2 \Phi \bigg( \frac{s_{n,\infty}}{s_{n,j}} \bigg) -1.
		$$
		Next, $s_{n,j} < s_{n,\infty}$ implies $\lambda_{n,j}(B) > \pi_n(B)$ and consequently,
		$$
		\|\lambda_{n,j} - \pi_n\|_{\mathrm{TV}} \geq \lambda_{n,j}(B) - \pi_n(B) = 2 \Phi \bigg( \frac{s_{n,\infty}}{s_{n,j}} \bigg) - 2 \Phi (1) = 2 \Big[ \Phi \big( (1 - \alpha_n^{2j})^{-1/2} \big) - \Phi(1) \Big].
		$$
		By a Taylor expansion, $(1 - \alpha_n^{2j})^{-1/2} = 1 + \alpha_n^{2j}/2 + \mathcal{O}(\alpha_n^{4j})$. Together with  $\Phi'(x)=(2\pi)^{-1/2}\exp(-x^2/2)=:\phi(x)$, the mean value theorem yields 
		$$
		\Phi \big( (1 - \alpha_n^{2j})^{-1/2} \big) - \Phi(1) = \phi(y) \big( (1 - \alpha_n^{2j})^{-1/2} - 1 \big) \geq \phi(1) \, \frac{\alpha_n^{2j}}{2} + \mathcal{O}(\alpha_n^{4j})
		$$
		for an appropriate intermediate value $y \in [1,(1 - \, \alpha_n^{2j})^{-1/2}]$. 
		By the reverse triangle inequality and the latter equation, we then write
		\begin{align}
			\| \mu_{n,j} - \pi_n\|_{\mathrm{TV}} &\geq \|\lambda_{n,j} - \pi_n\|_{\mathrm{TV}} - \| \mu_{n,j} - \lambda_{n,j}\|_{\mathrm{TV}} \notag \\
			&\geq \frac{\phi(1)}{2} \, \alpha_n^{2j} - \| \mu_{n,j} - \lambda_{n,j}\|_{\mathrm{TV}} \notag \\
			&\geq \frac{\phi(1)}{2} \, \alpha_n^{2j} - \frac{|x| \alpha_n^j}{\sqrt{2 \pi}s_{n,j}} , \label{eq:reverse}
		\end{align}
		where the last inequality follows from \eqref{eq:delta_mu}. Thus \eqref{eq:reverse}, combined with \eqref{eq:7.2}, leads to
		\begin{align*}
			\| \mu_{n,j} - \pi_n\|_{\mathrm{TV}} \geq \begin{cases}
				\phi(1)  \min \bigg\{ 1,  \frac{ |x| \alpha_n^j}{s_{n,\infty}}  \bigg\}, \qquad &\text{ if } \quad \frac{|x|}{\sqrt{2 \pi} s_{n,\infty}} \geq \frac{\phi(1)}{2} \alpha_n^j; \\
				\frac{\phi(1)}{2} \, \alpha_n^{2j}, \qquad  \hspace{20mm} &\text{ if } \quad  \frac{|x|}{\sqrt{2 \pi} s_{n,\infty}} < \frac{\phi(1)}{2} \alpha_n^j.
			\end{cases}
		\end{align*}
		In particular, for arbitrary $|x| \in \R_{\geq0}$ and $n \in \N$
		$$\| \mu_{n,j} - \pi_n\|_{\mathrm{TV}} \geq \max  \bigg\{  \phi(1)  \min \bigg\{ 1,  \frac{ |x| \alpha_n^j}{s_{n,\infty}}  \bigg\}, \frac{\phi(1)}{ 2} \alpha_n^{2j} \bigg\}.
		$$
		Using consecutively  $\max \{a+b\} \geq (a+b)/2$ for $a,b\geq 0$ and $\min \{a,b\} + c \geq \min \{a,b+c\}$ for $c \geq 0$, the latter display implies
		\begin{align}
			\| \mu_{n,j} - \pi_n\|_{\mathrm{TV}} \geq c   \min \bigg\{ 1,  \frac{ |x| \alpha_n^j}{s_{n,\infty}}   +   \alpha_n^{2j} \bigg\}, \label{eq:7.3}
		\end{align}
		with $c=\phi(1)/4$. As $0 \leq \| (\cdot) -  (\cdot) \|_{\mathrm{TV}} \leq 1$, the combination of \eqref{eq:7.1} and \eqref{eq:7.3} yields
		$$
		c   \min \bigg\{ 1,  \frac{ |x| \alpha_n^j}{s_{n,\infty}}   +   \alpha_n^{2j} \bigg\} \leq \| \mu_{n,j} - \pi_n\|_{\mathrm{TV}} \leq C   \min \bigg\{ 1,  \frac{ |x| \alpha_n^j}{s_{n,\infty}}   +   \alpha_n^{2j} \bigg\} .
		$$
		
		\smallskip
		\noindent
		To evaluate the asymptotic order in $n$ of $t_{\text{mix},n}$, we need to extract for any $0<\delta,\eta<1$ the order of the smallest possible index $j$ such that 
		$$
		k_n(\delta) \frac{\alpha_n^{j}}{s_{n,\infty}} + \alpha_n^{2j} \leq \eta,
		$$
		where $k_n(\delta) = \inf \{k > 0 : \pi_n \left( \left[ -k,k \right] \right) \geq 1 - \delta \}\asymp s_{n,\infty}$. As $\alpha_n=(1-\gamma/n^{\beta})$, one immediately reads off $t_{\text{mix},n}\asymp n^{\beta}/\gamma$.
	\end{proof}
\end{appendix}

\end{document}